\documentclass[11pt]{amsart}
\usepackage{amsmath}
\usepackage{amsfonts}
\usepackage{amssymb}
\usepackage{graphicx}
\usepackage{hyperref}
\usepackage{mathrsfs}
\usepackage{pb-diagram}
\usepackage{epstopdf}
\usepackage{amsmath,amsfonts,amsthm,enumerate,amscd,latexsym,curves}
\usepackage{bbm}
\usepackage[mathscr]{eucal}
\usepackage{epsfig,epsf,epic}
\usepackage{caption}
\usepackage{subcaption}
\usepackage{multirow}
\usepackage{color}
\usepackage[all]{xy}
\usepackage{fourier}
\usepackage{tikz-cd}
\usepackage{calc}
\usepackage{float}
\usepackage{enumitem}
\usepackage{comment}
\usepackage[normalem]{ulem}

\newtheorem{thm}{Theorem}[section]
\newtheorem{lem}[thm]{Lemma}

\newtheorem{cor}[thm]{Corollary}
\newtheorem{prop}[thm]{Proposition}
\newtheorem{dfn}[thm]{Definition}
\newtheorem{example}[thm]{Example}
\newtheorem{rmk}[thm]{Remark}

\numberwithin{equation}{section}

\newcommand{\vb}{\,|\,}

\newcommand{\Hom}{{\rm Hom}}

\newcommand{\cD}{\mathcal{D}}
\newcommand{\cC}{\mathcal{C}}

\newcommand{\cF}{\mathcal{F}}

\newcommand{\Z}{\mathbb{Z}}
\newcommand{\R}{\mathbb{R}}
\newcommand{\C}{\mathbb{C}}

\newcommand{\Tw}{\textup{Tw}}

\newcommand{\F}{\mathcal{F}}

\newcommand{\Cone}{\mathrm{Cone}}

\newcommand{\Perf}{\textup{Perf}}
\newcommand{\cW}{\mathcal{W}}
\newcommand{\W}{\mathcal{W}}
\newcommand{\len}{\mathrm{len}}

\newcommand{\Gr}{\mathrm{Gr}}

\begin{document}

	\title{On categorical entropy from the viewpoint of symplectic topology}
	
	\author{Hanwool Bae}
	\address[Hanwool Bae]{Center for Quantum Structures in Modules and Spaces, Seoul National University, Seoul, South Korea}
	\email{hanwoolb@gmail.com}
	
	\author{Dongwook Choa}
	\address[Dongwook Choa]{Department of Mathematics\\
		Korea Institute for Advanced Study	\\
		Seoul 02455, Korea}
	\email{dwchoa@kias.re.kr}
	
	\author{Wonbo Jeong}
	\address[Wonbo Jeong]{Department of Mathematical Sciences, Research Institute in Mathematics\\ Seoul National University\\ Gwanak-gu\\Seoul \\ South Korea}
	\email{wonbo.jeong@gmail.com}
	
	\author{Dogancan Karabas}
	\address[Dogancan Karabas]{Kavli Institute for the Physics and Mathematics of the Universe (WPI), The University of Tokyo Institutes for Advanced Study, The University of Tokyo, Kashiwa, Chiba 277-8583, Japan}
	\email{dogancan.karabas@ipmu.jp}
	
	\author{Sangjin Lee}
	\address[Sangjin Lee]{Center for Geometry and Physics\\ Institute for Basic Science (IBS)\\ Pohang 37673, Korea}
	\email{sangjinlee@ibs.re.kr}

	\begin{abstract}
		In this paper, motivated by symplectic topology, we explore categorical entropy and present two main results. The first result establishes a relation between categorical entropies of functors on a category and its localisation.
		Additionally, it demonstrates analogies between the notions of topological and categorical entropy.
		This result is then applied to symplectic topology, where we provide a method for calculating the categorical entropy of a functor on a (partially) wrapped Fukaya category, assuming that the functor is induced by a compactly supported symplectic automorphism.
		
		For the second main result of the paper, we observe the existence of natural examples of symplectic manifolds whose Fukaya categories satisfy a type of Floer-theoretic duality. Motivated by this observation, we prove that categorical entropy can be computed from the morphism spaces under the assumption of duality. The formula is similar to a result of \cite{Dimitrov-Haiden-Katzarkov-Kontsevich}, which is proven for the case of smooth and proper categories. 
	\end{abstract}
	
	\maketitle

	\section{Introduction}
	\label{section introduction}
	Let $\cC$ be a triangulated category, and let $\Phi: \cC \to \cC$ be an exact endofunctor on $\cC$. 
	Then, the pair $(\cC, \Phi)$ forms a categorical dynamical system. 
	Dimitrov, Haiden, Katzarkov, and Kontsevich \cite{Dimitrov-Haiden-Katzarkov-Kontsevich} 
	defined an invariant measuring the {\em complexity} of the dynamical system $(\cC, \Phi)$.
	The invariant is called {\em categorical entropy} of $\Phi$. 
	After \cite{Dimitrov-Haiden-Katzarkov-Kontsevich} introduced the notion, categorical entropy has drawn many interests, see \cite{Kikuta17, Fan18, Kikuta-Takahashi,  Ouchi20, Kikuta-Shiraishi-Takahashi20, Kikuta-Ouchi, Fan-Fu-Ouchi21, Ikeda21} for example.
	
	In the current paper, we study categorical entropy, motivated by the viewpoint of symplectic topology.
	Consider the following situation: 
	Let $M$ be a symplectic manifold, and let $\phi: M \to M$ be a symplectic automorphism.
	It is well-known that $\phi$ induces an auto-equivalence $\Phi$ on the Fukaya category of $M$. 
	Then, the Fukaya category of $M$ and the induced auto-equivalence $\Phi$ form a categorical dynamical system. 
	
	Many natural questions arise from this situation. 
	For example, the pair $(M, \phi)$ induces not only the categorical dynamical system, but also a topological dynamical system.
	What is the relation between the topological entropy of $\phi$ and the categorical entropy of $\Phi$?
	Another question concerns the properties of topological entropy known in the literature: Does categorical entropy exhibit similar properties?
	
	In the current paper, we study categorical entropy motivated by the following two specific questions:
	\begin{enumerate}
		\item[Question 1.] If $\phi$ is an exact symplectic (more precisely, Liouville) automorphism on a Weinstein manifold $W$, $\phi$ induces an auto-equivalence on each of the compact, wrapped, and partially wrapped Fukaya categories of $W$. 
		Then, what relations exist among the categorical entropies of the induced functors?
		\item[Question 2.] The categorical entropy is not easy to compute. 
		In the above situation, can we compute the categorical entropy of $\Phi$ {\em practically} from the symplectic topology?
	\end{enumerate}
	
	The first half of the paper, motivated by Question 1, studies the categorical entropies on a category and its localisation. 
	We note that Question 1 is motivated by symplectic topology, but everything in the first half, except the application subsection (Section \ref{section Entropy for wrapped Fukaya categories via compactly supported symplectomorphisms}), is written in the language of category theory. 
	The second half provides a way to compute the categorical entropy of a symplectic automorphism $\phi$ where $\phi$ is defined on a Weinstein manifold satisfying a {\em duality} condition.
	Moreover, it can be generalized to a triangulated category, under a {\em duality} condition.
	The rest of this introduction section contains more details on our results. 
	
	To answer Question 1, in the first half of the paper, we prove Theorem \ref{thm localisation} that is stated/proven in purely categorical way. 
	\begin{thm}[= Theorem \ref{thm:entropy-loc}]
		\label{thm localisation}
		Let $\cC$ be a triangulated (or pretriangulated $\mathcal{A}_\infty$/dg) category, and let $\cD$ be a triangulated full subcategory of $\cC$. 
		Let $\Phi_\cC:\cC \to \cC$ be a functor preserving $\cD$ so that there are induced functors
		\[\Phi_{\cD}:\cD\to\cD\quad\text{and}\quad\Phi_{\cC/\cD}: \cC/\cD \to \cC/\cD.\]
		If $h_t(\Phi)$ means the categorical entropy of a functor $\Phi$, the following inequalities hold:
		\begin{gather*}
			h_t\left(\Phi_{\cC/\cD}\right) \leq h_t\left(\phi_\cC\right) \leq \max \left\{h_t(\Phi_{\cC/\cD}),  h_t(\Phi_\cD)\right\}.
		\end{gather*}
	\end{thm}  
	
	We note that Theorem \ref{thm localisation} proves analogous relations between notions of topological and categorical entropies. 
	The analogous properties that topological and categorical entropy share could be another motive in proving Theorem \ref{thm localisation}, which does not rely on symplectic topology.
	
	To give an example of an analogy, let us assume that $X_1$ and $X_2$ are topological spaces equipped with continuous self mappings $f_i:X_i \to X_i$.
	If there exists a surjective map $F: X_1 \to X_2$ such that $F \circ f_1 = f_2 \circ F$, then it is known that 
	\[\text{the topological entropy of  } f_1 \geq \text{  the topological entropy of  } f_2.\]
	In category theory, one can see $\Phi_\cC, \Phi_{\cC/\cD}$ and the localisation functor $\cC\to\cC/\cD$ as the categorical analogues of $f_1, f_2$ and $F$, then one can see the analogy between the first inequality in Theorem \ref{thm localisation} and the above inequality.
	Similarly, the second inequality in Theorem \ref{thm localisation} also gives an analogy between two notions of entropies. 
	For more details, see Section \ref{subsection comparison with the topological entropy}.
	
	Since Question 1 is motivated by symplectic topology, it would be natural to apply Theorem \ref{thm localisation} in a symplectic setting. 
	As an application, we have the following theorem:
	\begin{thm}[= Theorem \ref{thm:entropy-w-pw}]
		\label{thm wrapped vs partially wrapped}
		If $h_t\left(\Phi_{\cW(W)}\right),h_{-t}\left(\Phi^{-1}_{\cW(W)}\right) \geq 0$ (the conditions automatically hold by definition when $t=0$), then
		\begin{gather}
			\label{eqn equal} h_t\left(\Phi_{\cW(W)}\right)=h_t\left({\Phi_{\cW(W,\Lambda)}}\right)=h_{-t}\left({\Phi^{-1}_{\cW(W,\Lambda)}}\right)=h_{-t}\left(\Phi^{-1}_{\cW(W)}\right),
		\end{gather}
		where $\Phi_{\cW(W)}$ (resp.\ $\Phi_{\cW(W,\Lambda)}$) denotes the functor induced from a compactly supported exact symplectic automorphism $\phi\colon W\to W$ on the wrapped Fukaya category $\cW(W)$ of $W$ (resp.\ partially wrapped Fukaya category $\cW(W,\Lambda)$ of $W$ with a stop $\Lambda$). The first equality in \eqref{eqn equal} holds without the condition $h_{-t}\left(\Phi^{-1}_{\cW(W)}\right) \geq 0$.
	\end{thm}
	
	Note that Equation \eqref{eqn equal} is a potential practical tool for computing the categorical entropy of $\Phi_{\cW(W)}$ for a general Weinstein manifold $W$ and a general compactly supported symplectic automorphism $\phi$. 
	To be more precise, let us recall that for a general $W$, it is always possible to choose a stop $\Lambda$ making the corresponding partially wrapped Fukaya category $\cW(W,\Lambda)$ smooth and proper.
	Then, by applying \cite[Theorem 2.6]{Dimitrov-Haiden-Katzarkov-Kontsevich}, the categorical entropy of the functor $\Phi_{\cW(W)}$ can be computed via {\em Lagrangian Floer theory}, as
	\begin{equation}\label{eqn smooth-proper}
		h_t\left(\Phi_{\cW(W)}\right)=h_t\left({\Phi_{\cW(W,\Lambda)}}\right)=\lim_{n\to\infty}\tfrac{1}{n}\log\sum_{k\in\Z}\dim \Hom^k_{\cW(W,\Lambda)}\left(G,\phi^n(G)\right)e^{-kt} ,
	\end{equation}
	where $G$ is a split-generator of $\cW(W,\Lambda)$, which is described in \cite{Chantraine-Rizell-Ghiggini-Golovko, Ganatra-Pardon-Shende18b}.
	
	The application of Theorem \ref{thm localisation} reveals the above equality between categorical entropy on wrapped and partially wrapped Fukaya categories, but the relation between categorical entropy on compact and (partially) wrapped Fukaya categories is still to be investigated. 
	To reveal the relation at least for some cases, in the second half, we restrict our attention to symplectic manifolds satisfying a duality condition between compact and wrapped Fukaya categories, which is similar to the Koszul duality condition. 
	For a formal statement of our duality condition, see Theorem \ref{thm generalized version}. 
	An example of symplectic manifold satisfying our duality condition (and having Koszul duality pattern) is a plumbing space of $T^*S^n$ where $n \geq 3$, along a tree. 
	See \cite{Etgu-Lekili17} for more details.

	For the symplectic manifolds with the duality condition, by comparing categorical entropy on Fukaya categories of two different flavors, we find equations answering Question 2 and also answering a part of Question 1. 
	The following theorem is the result:
	\begin{thm}[= Theorem \ref{thm general}]
		\label{thm generalized version}
		Let us assume that a Weinstein manifold $W$ has sets of Lagrangians $\{S_i\}_{1 \leq i \leq k}$ and $\{L_i\}_{1 \leq i \leq k}$ such that 
		\begin{itemize}
			\item $\{S_i\}_{1 \leq i \leq k}$ (resp.\ $\{L_i\}_{1 \leq i \leq k}$) generates the compact (resp.\ wrapped) Fukaya category of $W$,
			\item the morphism space in the wrapped Fukaya category $\Hom^*(S_i,L_j)$ is non-zero only if $i=j$ and is one-dimensional when $i=j$,
			\item $\bigoplus_{i,j} \Hom^* (S_i,S_j)$ is non-negatively graded, 
			\item  $\bigoplus_{i,j} \Hom^* (L_i,L_j)$ is non-positively graded, 
			\item $\dim \Hom^0 (S_i,S_i) = 1 = \dim \Hom^0 (L_i,L_i)$ for all $1\leq i \leq k$, and 
			\item $\Hom^0(S_i,S_j) = 0$ for all $i \neq j$.
		\end{itemize}
		Let $S := \bigoplus_i S_i, L:= \bigoplus_i L_i$.
		If $\phi$ is a Liouville automorphism on $W$, and if $\Phi_\mathcal{F}, \Phi_\mathcal{W}$ denote the functors on the compact and wrapped Fukaya category induced from $\phi$,
		then 
		the following hold:
		\begin{gather}
			\label{eqn result1}	h_t \left( \Phi_{\mathcal{F}}\right) = \lim_{m \to \infty} \tfrac{1}{m} \log  \sum_{k \in \Z} \dim \Hom^{n+k} \left(\phi^m(S), L\right) e^{kt}, \\
			\label{eqn result2}	h_t \left( \Phi_{\mathcal{W}}\right) = \lim_{m \to \infty} \tfrac{1}{m} \log  \sum_{k\in \Z} \dim \Hom^{n-k} \left(S, \phi^m(L)\right) e^{kt}, \\
			\label{eqn result3}	h_{t} \left(\Phi_{\F}\right) = h_{-t}\left(\Phi^{-1}_{\W}\right), h_{t} \left(\Phi_{\W}\right) = h_{-t}\left(\Phi^{-1}_{\F}\right).
		\end{gather}
		Moreover, if $h_t\left(\Phi_\mathcal{F}\right), h_t\left(\Phi_\mathcal{W}\right) \geq 0$ (these conditions automatically hold when $t=0$), and $\phi$ is compactly supported, then 
		\begin{gather}
			\label{eqn result4} h_{t} \left(\Phi_{\F}\right) = h_{t} \left(\Phi_{\W}\right) .
		\end{gather}		
	\end{thm}
	
	Equations \eqref{eqn result1} and \eqref{eqn result2} look similar to Equation \eqref{eqn smooth-proper}, but we do not require that the categories are smooth or proper. 
	In the proof of Theorem \ref{thm generalized version}, we employ the duality condition instead of smoothness and properness, and we can compute the categorical entropy as the exponential growth of the dimensions of morphism spaces. 
	
	We would like to point out that by combining Theorems \ref{thm wrapped vs partially wrapped} and \ref{thm generalized version}, one can compare entropies on compact, wrapped, and partially wrapped Fukaya categories. 
	
	\begin{rmk}
		\label{rmk category with duality}
		The proof of Equations \eqref{eqn result1}--\eqref{eqn result3} in Theorem \ref{thm generalized version} requires only the duality between $\cF$ and $\cW$. 
		Thus, one can generalize Equations \eqref{eqn result1}--\eqref{eqn result3} to a pair of categories $\left(\cC \supset \cD\right)$ satisfying the duality conditions. 
		An example of a pair of categories satisfying the duality condition is the pair of categories of the modules and finite-dimensional modules over $\Gamma_n T$, where $\Gamma_n T$ denotes Ginzburg $n$-Calabi-Yau dg algebra (without potential) over a tree $T$.
		The duality condition for the example can be proven from the quasi-equivalence between the wrapped (resp.\ compact) Fukaya category of the plumbing of $T^*S^n$ along $T$ and the module categories (resp.\ finite-dimensional module categories) over $\Gamma_n T$. 
	\end{rmk}
	
	Even though Theorem \ref{thm generalized version} gives formulas computing categorical entropy in a symplectic setting, the actual computation is not an easy task. 
	In Section \ref{section entropies of symplectic automorphisms of Penner type}, we give a construction of symplectic automorphisms whose categorical entropy are computed by simple linear algebra. 
	From the construction, one can find useful examples of symplectic automorphisms, for example, an element of a symplectic Torelli group having positive categorical entropy. 
	For more details, see Section \ref{subsubsection torelli group}.
	
	More specifically, we focus on symplectic automorphisms of {\em Penner type} (see Definition \ref{def Penner type},) which is constructed by a higher dimensional generalization of Penner's construction \cite{Penner} of pseudo-Anosov surface automorphisms. 
	In his thesis \cite{Lee}, the last-named author proved that symplectic automorphisms of the Penner type satisfy geometric stability. 
	We expected that one could easily compute the categorical entropy of the Penner type because of the geometric stability.
	The expectation turns out to be true: one can compute the categorical entropy by simple linear algebra.
	More precisely, we prove the following: 
	\begin{thm}[=Theorem \ref{thm:entropyofpennertype}]
		\label{thm penner type}
		Let $\phi$ be a symplectic automorphism of Penner type, defined on a plumbing space of $T^*S^n$ along a tree for some $n\geq 3$.
		Then, the Lagrangian Floer homology between $\phi(S)$ and $L$ defines a matrix $M_{\phi}$ such that the logarithm of the spectral radius of $M_\phi$ is the categorical entropy of $\phi$.
	\end{thm}
	
	\begin{rmk}
		After the first version of the current paper appearing on arXiv, the first and last named authors \cite{Bae-Lee22} proved that if $\phi$ is a compactly supported symplectic automorphism on a Weinstein manifold, then the categorical entropy of $\phi$ bounds the topological entropy of $\phi$ from below. 
		Thus, the answers to Question 2 given in the current paper can provide a lower-bound of the topological entropy of a corresponding symplectic automorphism. 
	\end{rmk}

	\subsection{Acknowledgment}
	\label{subsection acknowledgment}
	We appreciate Jongmyeong Kim for sharing useful information on categorical entropy. 
	
	During this work, Hanwool Bae was supported by the National Research Foundation of Korea(NRF) grant funded by the Korea government(MSIT) (No.2020R1A5A1016126), Dongwook Choa was supported by the KIAS individual grant (MG079401), Wonbo Jeong was supported by Samsung Science and Technology Foundation under project number SSTF-BA1402-52, Dogancan Karabas was supported by World Premier International Research Center Initiative (WPI), MEXT, Japan, and Sangjin Lee was supported by the Institute for Basic Science (IBS-R003-D1).

	\section{Twisted complex formulation of the categorical entropy}
	\label{section twisted complex formulation}

	In Section \ref{section twisted complex formulation}, we recall the notion of categorical entropy by \cite{Dimitrov-Haiden-Katzarkov-Kontsevich}, and redefine it using twisted complexes in the setting of $A_{\infty}$ and dg categories.
	While the original approach uses a single split-generator of a category, we show that the definition extends to multiple split-generators. See Proposition \ref{prp:split-entropy}.
	It has some practical advantages;
	see Sections \ref{section entropy for localised categories}, \ref{section entropies of products of Dehn twists}, or \ref{section entropies of symplectic automorphisms of Penner type}, for example. We will also compare the categorical entropies of functors with their opposites/inverses in Proposition \ref{prp:op and inv}.

	First, we set the notation. 
	Let $\cC$ be a (strictly unital or cohomologically unital) $A_{\infty}$ or dg category, which is $\Z$-graded (or $\Z/2$-graded) and $k$-linear for some coefficient field $k$. We write $\hom^*_{\cC}(A,B)$ for the morphism complex of the objects $A,B\in\cC$  with the differential $d$, and $\Hom^*_{\cC}(A,B)$ for the cohomology of $\hom^*_{\cC}(A,B)$. Equivalences of $A_{\infty}$/dg categories are quasi-equivalences, which we denote by $\simeq$. We will also use the symbol $\simeq$ when two objects (resp.\ morphisms) are homotopy equivalent (resp.\ homotopic). We write $A[n]$ for the $n$-shift of an object $A\in\cC$, which can be thought as $k[n]\otimes A$.
	We write $\{ G_1, \ldots, G_m \}_{\cC}$ (or $\{ G_1, \ldots, G_m \}$ if the category $\cC$ is clear from the context) for the full $A_{\infty}$/dg subcategory of $\cC$ consisting $G_1,\ldots,G_m\in\cC$.
	The reader can refer to \cite{Seidel08} for more details.
	
	Now, we recall the notion of twisted complex.
	\begin{dfn}\label{dfn:twisted-complex}
		Let $\cC$ be a dg category.
		\begin{enumerate}
			\item A {\em twisted complex $\tilde K=\left[(K_i[d_i])_{i=1}^n,(f_{ij})_{1\leq j<i\leq n}\right]$ in $\cC$} consists of objects $K_i\in\cC$, shifts $d_i\in\Z$, and degree $1$ morphisms
			\[f_{ij}\colon K_j[d_j]\to K_i[d_i]\]
			for $i>j$, satisfying
			\begin{equation}\label{eq:twisted-complex-mc}
				df_{ij}+\sum_{i>l>j} f_{il}\circ f_{lj}=0 .
			\end{equation}
			If we define
			\[K:=K_1[d_1]\oplus K_2[d_2]\oplus\ldots\oplus K_n[d_n],\]
			and the strictly lower triangular degree 1 matrix
			\[f:=(f_{ij})\colon K\to K ,\]
			then Equation \ref{eq:twisted-complex-mc} becomes
			\[df + f\circ f=0 .\]
			\item A {\em (homogeneous) morphism $\tilde\alpha\colon [(K_i[d_i]),(f_{ij})]\to [(L_i[e_i]),(g_{ij})]$} of twisted complexes is a (homogeneous) morphism $\alpha\colon K\to L$ with the grading $|\tilde\alpha|:=|\alpha|$ and the differential
			\[d\tilde\alpha:=d\alpha + g\circ\alpha -(-1)^{|\alpha|}\alpha\circ f.\]
			Hence, the morphism space between the twisted complexes $\tilde K:=[(K_i[d_i]),(f_{ij})]$ and $\tilde L:=[(L_i[e_i]),(g_{ij})]$ is given by
			\begin{equation}\label{eq:hom-space-tw}
				\hom^*(\tilde K,\tilde L)=\bigoplus_{i,j} \hom^*(K_i,L_j)[e_j-d_i]
			\end{equation}
			as a graded vector space, equipped with the above differential.
			
			\item The {\em $d$-shift} of $[(K_i[d_i]),(f_{ij})]$ is $[(K_i[d_i+d]),(f_{ij}[d])]$.
			\item The {\em cone of a closed degree zero morphism $\tilde\alpha$}, denoted by $\Cone(\tilde\alpha)$, is the twisted complex
			\[\left[(K_i[d_i+1], L_j[e_j]), \begin{pmatrix}f[1] & 0\\ -\alpha & g\end{pmatrix}\right] .\]
			\item A closed degree zero morphism $\tilde\alpha$ is a {\em homotopy equivalence} if  $\Cone(\tilde\alpha)\simeq 0$, i.e., $1_{\Cone(\tilde\alpha)}=d\tilde h$ for some $\tilde h$.
			\item We write $\Tw(\cC)$ for {\em the dg category of twisted complexes in $\cC$}. 
			\item For any $\tilde L\in\Tw(\cC)$, we call $\tilde K=[(K_i[d_i]),(f_{ij})]$ {\em a twisted complex for $\tilde L$} if $\tilde K$ and $\tilde L$ are homotopy equivalent in $\Tw(\cC)$.
		\end{enumerate}
	\end{dfn}

	\begin{rmk}
		If $\cC$ is an $A_{\infty}$-category, we can similarly define twisted complexes in $\cC$ and $\Tw(\cC)$. However, Equation \ref{eq:twisted-complex-mc} becomes
		\[\sum_{r=1}^{\infty}\mu^r_{\Sigma(\cC)}(f,\ldots,f)=0\]
		where $\Sigma(\cC)$ is the additive enlargement of $\cC$, which is obtained by enlarging $\cC$ by finite direct sums and shifts, and $\mu^r_{\Sigma(\cC)}$ are $A_{\infty}$-operations of $\Sigma(\cC)$. See \cite{Seidel08} for more details.
	\end{rmk}
	
	\begin{rmk}
		We will mostly use uppercase letters with tilde for twisted complexes, and lowercase letters with tilde for the morphisms between twisted complexes. When we write $\tilde K\in\Tw(\cC)$, it will be understood that $\tilde K$ is an explicit twisted complex given by
		\[\tilde K=\left[(K_i[d_i])_{i=1}^n,(f_{ij})_{1\leq j<i\leq n}\right]\]
		for some $K_i\in\cC$, $d_i\in\Z$, and degree 1 morphisms $f_{ij}\colon K_j\to K_i$ for $i>j$ satisfying (\ref{eq:twisted-complex-mc}). Note that $\tilde K$ contains shifts $d_i$'s as extra information.
	\end{rmk}
	
	\begin{rmk}
		\mbox{}
		\label{rmk iterated cone}
		\begin{enumerate}
			\item \label{item:notation} Another notation for the twisted complex
			\[\tilde K=\left[(K_i[d_i])_{i=1}^n,(f_{ij})_{1\leq j<i\leq n}\right]\]
			is
			\begin{equation}\label{eq:twisted-complex-alternative}
				[\begin{tikzcd}
					K_1[d_1]\ar[r,"f_{21}"]\ar[rr,bend right=20,"f_{(n-2)1}"]\ar[rrr,bend right=25,"f_{(n-1)1}"]\ar[rrrr,bend right=30,"f_{n1}"] & \ldots\ar[r,"f_{(n-2)(n-3)}"] & K_{n-2}[d_{n-2}]\ar[r,"f_{(n-1)(n-2)}"]\ar[rr,bend right=20,"f_{n(n-2)}"'] & K_{n-1}[d_{n-1}]\ar[r,"f_{n(n-1)}"] & K_n[d_n]
				\end{tikzcd}] .
			\end{equation}
			
			\item \label{item:iterated}A twisted complex $\tilde K\in\Tw(\cC)$ can be seen as an iterated cone in $\Tw(\cC)$. Order of taking cone can be chosen freely after properly shifting $K_i$'s. One example of iterated cones corresponding to $\tilde K$ in (\ref{eq:twisted-complex-alternative}) is
			\[\Cone(K_1[d_1-1]\to\ldots\to\Cone(K_{n-2}[d_{n-2}-1] \xrightarrow{(f_{(n-1)(n-2)}\; f_{n(n-2)})} \Cone(K_{n-1}[d_{n-1}-1] \xrightarrow{f_{n (n-1)}} K_n[d_n]))) .\]
			
			\item \label{item:twisted} A twisted complex $\tilde K\in\Tw(\cC)$ can be expressed as $\tilde K=\left[(\tilde K_i[d_i])_{i=1}^n,(\tilde f_{ij})_{1\leq j<i\leq n}\right]$,
			where $\tilde K_i\in\Tw(\cC)$, and the morphisms $\tilde f_{ij}\colon\tilde K_j[d_j]\to\tilde K_i[d_i]$ for $i>j$ satisfy
			\begin{equation*}
				d\tilde f_{ij}+\sum_{i>l>j} \tilde f_{il}\circ \tilde f_{lj}=0 .
			\end{equation*}
			In this case, we regard this twisted complex as in (\ref{eq:twisted-complex-alternative}), where each $K_i$ is replaced by the actual twisted complex $\tilde K_i=[(L^i_{p}[e^i_p]),(g^i_{pq})]\in\Tw(\cC)$, and each $f_{ij}$ is replaced by the arrows between twisted complexes coming from $\tilde f_{ij}$.
		\end{enumerate}
	\end{rmk}
	
	
	\begin{rmk}
		Given a twisted complex
		\[\tilde K=\left[(K_i[d_i])_{i=1}^n,(f_{ij})_{1\leq j<i\leq n}\right]\in\Tw(\cC),\]
		its (essential) image under an $A_{\infty}$ (or dg) functor $\Phi\colon\Tw(\cC)\to\Tw(\cD)$ is given by
		\[\Phi(\tilde K)=\left[(\Phi(K_i)[d_i])_{i=1}^n,(\tilde g_{ij})_{1\leq j<i\leq n}\right]\in\Tw(\cD),\]
		for some $\tilde g_{ij}\colon \Phi(K_j)[d_j]\to\Phi(K_i)[d_i]$, where $\Phi(\tilde K)$ is regarded as in Remark \ref{rmk iterated cone}\eqref{item:twisted}. See \cite[Section (3m)]{Seidel08} for more details. This will be useful later when we study twisted complexes under the repetitive application of an $A_{\infty}$/dg functor.
	\end{rmk}
	
	Next, we will define the length of a twisted complex, which will be used in the definition of categorical entropy.
	\begin{dfn}\label{dfn:length}
		Let $\cC$ be an $A_{\infty}$ (or dg) category. Consider a twisted complex
		\[\tilde K=\left[(K_i[d_i])_{i=1}^n,(f_{ij})_{1\leq j<i\leq n}\right]\in\Tw(\cC) .\]
		\begin{enumerate}
			\item {\em Components of $\tilde K$} is the collection
			\[\{K_i[d_i]\vb i=1,\ldots,n\} .\]
			Given a full dg subcategory $\cD$ of $\cC$, {\em $\cD$-components of $\tilde K$} is the collection
			\[\{K_i[d_i]\vb K_i\in\cD\} .\]
			\item For a given $t\in\R$, {\em the length of $\tilde K$ (at $t$)} is defined as
			\[\len_t\tilde K:=\sum_{i=1}^n e^{d_i t}.\]
			If $t=0$, then the length of $\tilde K$ is just $n$. Given full dg subcategory $\cD$ of $\cC$, {\em the length of $\tilde K$ (at $t$) with respect to $\cD$} is defined as
			\[\len_{t,\cD}\tilde K:=\sum_{K_i\in\cD} e^{d_i t}.\]
		\end{enumerate}
		When $\cD$ consists of a single object $D$, we will write $D$-components of $\tilde K$ (resp.\ $\len_{t,D}\tilde K$) for $\cD$-components of $\tilde K$ (resp.\ $\len_{t,\cD}\tilde K$) in short. When $t=0$, we will sometimes just write $\len\,\tilde K$ (resp.\ $\len_{\cD}\tilde K$) for $\len_t\tilde K$ (resp.\ $\len_{t,\cD}\tilde K$).
	\end{dfn}
	
	\begin{rmk}
		If $\cC$ is $\Z/2$-graded, the length of a twisted complex only makes sense if $t=0$. The same will be true for the complexity (Definition \ref{dfn:complexity}) and the categorical entropy (Definition \ref{dfn:entropy}).
	\end{rmk}
	
	Now, we define the notion of (split-closed) pretriangulated $A_{\infty}$/dg categories and (split-)generators, in order to define the {\em categorical entropy} later.
	
	\begin{dfn}
		\label{dfn:pretriangulated-dg}
		Let $\cC$ be an $A_{\infty}$ (or dg) category.
		\begin{enumerate}
			\item An $A_{\infty}$ (or dg) category $\cC$ is {\em pretriangulated}, if there is a quasi-equivalence $\cC\simeq \Tw(\cC)$.
			Hence, $\Tw(\cC)$ is also called {\em the pretriangulated closure} of $\cC$.
			
			\item We write $\Perf(\cC)$ for {\em the split-closed pretriangulated closure} of $\cC$, which is obtained by splitting direct summands in $\Tw(\cC)$.
			\item We say $G_1,\ldots,G_m\in\cC$ {\em generate} (resp.\ \em{split-generate}) $\cC$ if 
			\[\cC\subset\Tw(\{ G_1, \ldots, G_m \}) \hspace{0.1cm} \left(\textrm{resp.} \cC\subset\Perf(\{ G_1, \ldots, G_m \}) \right).\]
			We call $G_1, \ldots, G_m$ \em{generators} (resp.\ \em{split-generators}) of $\cC$.
		\end{enumerate}
	\end{dfn}
	
	Note that the length of a twisted complex is not invariant under homotopy equivalence, but the following definition is an invariant.
	
	\begin{dfn}\label{dfn:complexity}
		Let $\cC$ be a pretriangulated $A_{\infty}$ (or dg) category, and $G_1,\ldots,G_m,C\in\cC$.
		Given $t\in\R$, the {\em complexity of $C$ (at $t$) with respect to $G_1,\ldots,G_m$} is defined as
		\[
		\delta_t(G_1,\ldots,G_m;C):=\inf\;\{\len_t\tilde K\vb \tilde K\in\Tw\left(\{ G_1,\ldots,G_m\}\right) \text{  and } \tilde K\simeq C\oplus C' \text{ for some }C'\in\cC\}.
		\]
	\end{dfn}
	Note that if $C$ is not split-generated by $G_1,\ldots,G_m$, then $\delta_t(G_1,\ldots,G_m;C)=\infty$. Also, when $t=0$, the complexity is zero if and only if $C\simeq 0$.
	
	\begin{rmk}
		\label{rmk:tower=twistedcomplex}
		The complexity $\delta_t(G;C)$ in \cite{Dimitrov-Haiden-Katzarkov-Kontsevich} is defined on a triangulated category $\mathcal T$ using a (right) Postnikov system
		\[
		\begin{tikzcd}[column sep=tiny]
			0 \ar[rr] && C_{n-1}\ar[rr] \ar[dl]&&\ldots \ar[rr]&& C_2 \ar[rr] && C_1 \ar[rr] \ar[dl] && C\oplus C'\ar[dl]\\
			& G[d_n] \ar[ul, dashed, "+1"]&&&&&&G[d_2] \ar[ul, dashed, "+1" ]&&G[d_1] \ar[ul, dashed, "+1"]&&
		\end{tikzcd}
		\]
		rather than a twisted complex. In the $A_{\infty}$/dg enhancement $\cC$ of $\mathcal T$ (i.e., $H^0(\cC)\simeq\mathcal T$), any such system can be seen as an iterated cone of $G[d_i]$'s, which provides a twisted complex quasi-isomorphic to $C\oplus C'$ as in Remark \ref{rmk iterated cone}\eqref{item:iterated} (after setting $K_i[d_i]=G[d_i]$), and vice versa.
	\end{rmk}
	
	We will now present the definition of categorical entropy for functors on $A_{\infty}$/dg categories, which is originally given in \cite{Dimitrov-Haiden-Katzarkov-Kontsevich} on triangulated categories.
	\begin{dfn}\label{dfn:entropy}
		Let $\cC$ be a pretriangulated $A_{\infty}$ (or dg) category split-generated by a single object $G$, and $\Phi\colon \cC\to\cC$ be an $A_{\infty}$ (or dg) functor. 
		Then, for a given $t\in\R$, {\em the categorical entropy of $\Phi$} is defined as
		\[h_t(\Phi):=\lim_{n\to\infty}\frac{1}{n}\log\delta_t\left(G;\Phi^n(G)\right)\in \{-\infty\}\cup \R .\]
	\end{dfn}
	The categorical entropy of $\Phi$ is well-defined and independent of the choice of the split-generator $G$ of $\cC$, see \cite{Dimitrov-Haiden-Katzarkov-Kontsevich}. We also note that $h_{t}(\Phi)=-\infty$ if $\Phi^n(G)\simeq 0$ for some $n$. 
	
	\begin{rmk}
		\mbox{}
		\begin{enumerate}
			\item We can also define the categorical entropy for an $A_{\infty}$/dg functor $\Phi\colon\cC\to\cC$ on any $A_{\infty}$/dg category $\cC$ by considering its pretriangulated closure $\Tw(\cC)$ and the induced functor $\Phi\colon\Tw(\cC)\to\Tw(\cC)$.
			\item Since the notion of categorical entropy, Definition \ref{dfn:entropy}, is defined for $A_{\infty}$/dg functors on $A_{\infty}$/dg categories with a split-generator, we assume that every $A_{\infty}$/dg category admits a split-generator throughout this paper in the context of categorical entropy.
			\item As we remarked earlier, if $\cC$ is $\Z/2$-graded, the categorical entropy only makes sense when $t=0$. However, we can consider $\cC$ as a 2-periodic $\Z$-graded category to be able to talk about the categorical entropy when $t\neq 0$, in which case, $h_{t\neq0}(\Phi)=-\infty$ for any $\Phi$.
		\end{enumerate} 
	\end{rmk}
	
	Definition \ref{dfn:entropy} defines the categorical entropy using the length with respect to a split-generator. 
	The following proposition shows that one can also compute the categorical entropy from the length with respect to a finite collection of split-generators.
	
	\begin{prop}\label{prp:split-entropy}
		Let $\cC$ be a pretriangulated $A_{\infty}$ (or dg) category split-generated by $G_1,\ldots,G_m$, and define
		$G:=G_1\oplus\ldots\oplus G_m$. 
		Let $\Phi\colon \cC\to\cC$ be an $A_{\infty}$ (or dg) functor. Then, for a given $t\in\R$, the categorical entropy of $\Phi$ can be given by
		\[h_t(\Phi)=\lim_{n\to\infty}\frac{1}{n}\log\delta_t\left(G_1,\ldots,G_m;\Phi^n(G)\right).\]
	\end{prop}
	
	\begin{proof}
		First, note that $G$ is a split-generator for $\cC$. Let $\tilde K\in\Tw(\{G_1,\ldots,G_m\})$ be an arbitrary twisted complex for $\Phi^n(G)\oplus A$ for some $A\in\cC$. Then one can replace each $G_i$ in $\tilde K$ with $G$ to get a twisted complex $\tilde K'\in\Tw(\{G\})$ for $\Phi^n(G)\oplus A\oplus B$ for some $B\in\cC$. Note that
		$\len_t \tilde K' = \len_t \tilde K$.
		This shows
		\[\delta_t\left(G;\Phi^n(G)\right)\leq \delta_t\left(G_1,\ldots,G_m;\Phi^n(G)\right).\]
		
		On the other hand, let $\tilde L\in\Tw(\{G\})$ be an arbitrary twisted complex for $\Phi^n(G)\oplus C$ for some $C\in\cC$. Then if we replace each $G$ in $\tilde L$ with the equivalent twisted complex $[G_1\xrightarrow{0}\ldots\xrightarrow{0}G_m]$, we get a twisted complex $\tilde L'\in\Tw(\{G_1,\ldots,G_m\})$ for $\Phi^n(G)\oplus C$. Note that
		$\len_t\tilde L' = m\times \len_t \tilde L$.
		This shows
		\[\delta_t(G_1,\ldots,G_m;\Phi^n(G))\leq m\delta_t(G;\Phi^n(G)) .\]
		By applying $\lim_{n\to\infty}\frac{1}{n}\log$ to the inequalities
		\[\delta_t(G;\Phi^n(G))\leq \delta_t(G_1,\ldots,G_m;\Phi^n(G))\leq m\delta_t(G;\Phi^n(G)),\]
		we get
		\[h_t(\Phi)\leq \lim_{n\to\infty}\frac{1}{n}\log\delta_t(G_1,\ldots,G_m;\Phi^n(G))\leq h_t(\Phi).\]
	\end{proof}
	We finish this section with the following proposition regarding an opposite and an inverse functor. 
	\begin{prop}
		\label{prp:op and inv}
		Let $\cC$ be a pretriangulated $A_{\infty}$ (or dg) category, and $\Phi\colon\cC\to\cC$ be an $A_{\infty}$ (or dg) functor. For any $t\in\R$, the following hold;
		\begin{enumerate}
			\item \label{item:op} $h_t(\Phi)=h_{-t}(\Phi^{\mathrm{op}})$, where $\Phi^{\mathrm{op}}$ is an induced functor on $\cC^{\mathrm{op}}$. 
			\item \label{item:inv} If $\Phi$ is an auto-equivalence, and $\cC$ is saturated (i.e., is smooth, proper, and admits a split-generator), then $h_{t}(\Phi^{-1})=h_{-t}(\Phi)$.
		\end{enumerate}
	\end{prop}
	
	\begin{proof}
		Let $G$ be a split-generator of $\cC$. If $\tilde K= \left[(G[d_1],\ldots,G[d_n]), (f_{ij})\right]\in\Tw(\{G\}_{\cC})$
		is a twisted complex for $\Phi^n(G)\oplus C$ for some $C\in\cC$, then
		\begin{align*}
			\tilde K'=\left[(G^{\mathrm{op}}[-d_n],\ldots, G^{\mathrm{op}}[-d_1]), ((f_{n+1-j, n+1-i})^{\mathrm{op}})\right]\in\Tw(\{G^{\mathrm{op}}\}_{\cC^{\mathrm{op}}})
		\end{align*}
		is a twisted complex for $(\Phi^n(G)\oplus C)^{\mathrm{op}}=(\Phi^{\mathrm{op}})^n(G^{\mathrm{op}})\oplus C^{\mathrm{op}}$, where $(G[d_i])^{\mathrm{op}}$ is viewed as the $(-d_i)$-shift of $G^{\mathrm{op}}$ in $\cC^{\mathrm{op}}$. Since $G^{\mathrm{op}}$ split-generates $\cC^{\mathrm{op}}$, this implies the first statement. 
		
		Next, assume $\cC$ is saturated. Then it is known by \cite{Dimitrov-Haiden-Katzarkov-Kontsevich} that 
		\begin{align*}
			h_t(\Phi) &= \lim_{n\to\infty}\frac{1}{n}\log \sum_k \left( \dim(\Hom^k_\cC(G, \Phi^n(G))) \right)e^{-kt} .
		\end{align*}
		Since $\cC^{\mathrm{op}}$ is also saturated, we have
		\begin{align*}
			h_{-t}(\Phi) = h_t(\Phi^{\mathrm{op}})& = \lim_{n\to\infty}\frac{1}{n}\log \sum_k \left( \dim(\Hom^k_{\cC^{\mathrm{op}}}(G^{\mathrm{op}}, (\Phi^{\mathrm{op}})^n(G^{\mathrm{op}})))\right)e^{-kt}\\
			&=\lim_{n\to\infty}\frac{1}{n}\log \sum_k \left(\dim(\Hom^k_\cC(\Phi^n(G), G))\right)e^{-kt}\\
			&=\lim_{n\to\infty}\frac{1}{n}\log \sum_k \left( \dim(\Hom^k_\cC(G,(\Phi^{-1})^n(G)))\right)e^{-kt} = h_t(\Phi^{-1}).
		\end{align*}
	\end{proof}
	
	\section{Categorical entropy of functors descended to quotients}
	\label{section entropy for localised categories}
	The main theorem of Section \ref{section entropy for localised categories}, Theorem \ref{thm:entropy-loc}, is a comparison of categorical entropies of $A_{\infty}$/dg functors on an $A_{\infty}$/dg category, its full $A_{\infty}$/dg subcategory, and their quotient.
	
	In Section \ref{subsection comparison with the topological entropy}, we explain a motivation coming from topological entropy.
	Sections \ref{subsection localised category} and \ref{subsection Theorem entropy-loc} are devoted to the proof of the Theorem \ref{thm:entropy-loc}. 
	In Section \ref{subsection when D is admissible}, we investigate the case of categories with admissible subcategories and get a comparison result, Corollary \ref{cor:orthogonal}, which follows from Theorem \ref{thm:entropy-loc}. 
	In Section \ref{section Entropy for wrapped Fukaya categories via compactly supported symplectomorphisms}, we apply Theorem \ref{thm:entropy-loc} to symplectic topology and compare the categorical entropies of functors on wrapped and partially wrapped Fukaya categories in Theorem \ref{thm:entropy-w-pw}.
	It provides a way calculating categorical entropy of symplectic automorphisms on wrapped Fukaya categories. 
	
		\subsection{Motivation: Comparison with properties of the topological entropy}
	\label{subsection comparison with the topological entropy}
	As its name suggests, it is natural to expect that the categorical entropy and other notions of entropy share similar properties. 
	Since our motivational setting appears in (symplectic) topology, it would be natural to ask whether the categorical entropy satisfies analogous properties of those of the {\em topological} entropy.
	Especially, in the section, we consider the following three properties of topological entropy:
	\begin{enumerate}
		\item[(i)] Let $f:X \to X$ be a continuous self-mapping and let us assume that there exists a subset $Y \subset X$ such that the restriction of $f$ onto $Y$ is a self-mapping on $Y$. 
		Then, 
		\[\text{the topological entropy of  } f \geq \text{the topological entropy of  } f|_Y.\]  
		\item[(ii)] Let $f:X \to X$ be a continuous self-mapping, and let us assume that there exist two subsets $Y_1$ and $Y_2$ of $X$ such that $Y_1 \cup Y_2 = X$ and $f(Y_i) \subset Y_i.$
		It is known that
		\[\text{the topological entropy of  } f = \max_{i=1, 2} \left\{\text{the topological entropy of  } f|_{Y_i}\right\}.\]

		\item[(iii)] Take two continuous self mappings of compact space, say $f_i : X_i \to X_i$ for $i =1, 2$, and let $F:X_1 \to X_2$ be a continuous mapping such that $F \circ f_1 = f_2 \circ F$. 
		If $F$ is onto, then
		\[\text{the topological entropy of  } f_1 \geq \text{  the topological entropy of  } f_2.\]
	\end{enumerate}
	For more details on the listed properties, we refer the reader to \cite[Section 1.6]{Gromov2}.
	
	We state specific questions about the categorical entropy, which are motivated by the above properties (i)--(iii).
	\begin{enumerate}
		\item[(i')] Let $\Phi: \cC \to \cC$ be an endofunctor and $\cD \subset \cC$ be a subcategory preserved by $\Phi.$ Does the following inequality hold?
		\[ h_t(\Phi) \geq h_t(\Phi \vert_\cD).\]
		\item[(ii')] Further suppose $\cC$ admits a semi-orthogonal decomposition $\cC =\langle \cD^\perp, \cD \rangle$ and $\Phi$ preserves each components. Does the following equality hold?
		\[h_t(\Phi) = \max \left\{ h_t(\Phi \vert_\cD), h_t(\Phi\vert_{\cD^\perp}) \right\}.\]
		\item[(iii')] In general, let $\cC/\cD$ denote the quotient category of $\cC$ by $\cD$ and $\Phi_{\cC/\cD}$ denote the induced functor of $\Phi$ on the quotient. Does the following inequality hold?
		\[ h_t(\Phi)\geq h_t(\Phi_{\cC/\cD}).\]
	\end{enumerate}
	
	In Section \ref{section entropy for localised categories}, we study the above questions and the following are our answers to the questions. 
	The reasoning behind the answers will be explained in the upcoming subsections.
	\begin{enumerate}
	\item[(I)] We do not have an affirmative answer for (i') generally. However, for the case of admissible $\mathcal{D}$, we prove that the answer to (i') is yes.
	See Section \ref{subsection when D is admissible}.
	\item[(II)] The answer to (ii') is always yes. See Corollary \ref{cor:orthogonal}.
	\item[(III)] The answer to (iii') is always yes. See Theorem \ref{thm:entropy-loc}. The relevant inequality is strengthened to
	\[ h_t(\Phi_{\cC/\cD}) \leq h_t(\Phi) \leq \max\{h_t(\Phi_{\cC/\cD}),h_t(\Phi \vert_{\cD})\}.\]
	\end{enumerate} 
	
	Before moving on to the main part of \ref{section entropy for localised categories}, we would like to summarize our conventions. 
	See the remark below:
	\begin{rmk}\label{rmk:motivation}
		\mbox{}
	\begin{enumerate}
		\item 
		\label{rmk a-inf to dg}
		We will be mostly working with dg categories and dg functors, but the results and proofs will apply to $A_{\infty}$-categories and $A_{\infty}$-functors because of the following argument: Any $A_{\infty}$-category is quasi-equivalent to a dg category by Yoneda embedding. For any dg categories $\cC$ and $\cD$, we also have (see \cite{Faonte} and \cite{Canonaco-Ornaghi-Stellari})
		\[\Hom_{A_{\infty}}(\cC,\cD)\simeq \mathrm{RHom}_{\mathrm{dg}}(\cC,\cD)\simeq\Hom_{\mathrm{dg}}(\cC',\cD)\]
		where $\Hom_{A_{\infty}}(\cC,\cD)$ (resp.\ $\Hom_{\mathrm{dg}}(\cC,\cD)$) is the space of $A_{\infty}$ (resp.\ dg) functors from $\cC$ to $\cD$, $\mathrm{RHom}_{\mathrm{dg}}(\cC,\cD)$ is the derived hom-space, and $\cC'$ is a cofibrant dg category quasi-equivalent to $\cC$. This means that any $A_{\infty}$-functor between $A_{\infty}$-categories can be considered as a dg functor between dg categories.
	\item 
	\label{rmk:triangulated-cat}
	Moreover, all the results in Section \ref{section twisted complex formulation} and Section \ref{section entropy for localised categories}(except Proposition \ref{prp:op and inv}\eqref{item:inv}) hold for triangulated categories and exact functors between them. When $\cC$ and $\cD$ are triangulated categories, $\cC/\cD$ denotes the Verdier quotient. The proofs can be modified by replacing twisted complexes by iterated cones as in Remark \ref{rmk iterated cone}\eqref{item:iterated}.
	We work with pretriangulated $A_{\infty}$/dg categories instead of triangulated categories for the notational convenience.
	\end{enumerate}
	\end{rmk}

	\subsection{Quotient of an $A_{\infty}$/dg category}
	\label{subsection localised category}
	First, we define the notion of the dg quotient of a dg category.
	
	\begin{dfn}[\cite{Drinfeld}]
		Let $\cC$ be a dg category.
		\begin{enumerate}
			\item If $\cD$ is a full dg subcategory of $\cC$, then the {\em dg quotient $\cC/\cD$}, defined up to quasi-equivalence, is obtained from $\cC$ by adding a degree $-1$ morphism $\varepsilon_D\colon D\to D$ for each $D\in\cD$ freely (in algebra level), such that
			\[d\varepsilon_D=1_D .\]
			\item {\em The dg localisation functor} 
			\[l\colon\cC\to\cC/\cD\]
			is the dg functor sending everything to itself. 
		\end{enumerate}
	\end{dfn}
	It is easy to check that $D \simeq 0$ in $\cC/\cD$ for all $D \in \cD \subset \cC$.
	
	\begin{rmk}	
		We can define the $A_{\infty}$ quotient of an $A_{\infty}$-category and the $A_{\infty}$-localisation functor using the dg quotient and dg localisation functor via Remark \ref{rmk:motivation}\eqref{rmk a-inf to dg}. For a more direct approach, see \cite{Lyubashenko-Ovsienko}.
	\end{rmk}
	
	
	
	
	Next proposition shows that the ($A_{\infty}$/dg) quotient of a pretriangulated $A_{\infty}$/dg category is again pretriangulated.
	
	\begin{prop}	
		Let $\cC$ be a pretriangulated $A_{\infty}$ (or dg) category, and $\cD$ be its full $A_{\infty}$ (or dg) subcategory. Then the quotient $\cC/\cD$ is also pretriangulated, and quasi-equivalent to $\cC/\Tw(\cD)$.
	\end{prop}
	
	\begin{proof}
		Assume $\cC$ and $\cD$ are dg categories for simplicity as in Remark \ref{rmk:motivation}\eqref{rmk a-inf to dg}. Observe that $\cC/\Tw(\cD)$ is the dg quotient of $\cC/\cD$ by $\Tw(\cD)$. In $\cC/\cD$, we have $D\simeq 0$ for all $D\in\cD\subset\cC$. Therefore, $E \simeq 0$ for all $E\in\Tw(\cD)\subset\cC$. Hence, taking quotient by $\Tw(\cD)$ does not have any effect, and we get
		\[\cC/\cD\simeq \cC/\Tw(\cD)\simeq\Tw(\cC/\cD),\]
		where the second equivalence is by \cite{Drinfeld}. This in particular shows that $\cC/\cD$ is pretriangulated.
	\end{proof}
	
	In Section \ref{section entropy for localised categories}, we compare the categorical entropies of functors $\Phi_{\cC},\Phi_{\cD},\Phi_{\cC/\cD}$ on $\cC, \cD$ and $\cC/\cD$, respectively, described in the following definition.
	\begin{dfn}\label{dfn:loc-induced-functor}
		
		Let $\cC$ be a pretriangulated $A_{\infty}$ (or dg) category, and $\cD$ be its full pretriangulated $A_{\infty}$ (or dg) subcategory. Let $\Phi_{\cC}\colon \cC\to \cC$ be an $A_{\infty}$ (or dg) functor satisfying $\Phi_{\cC}(\cD)\subset\cD$. Then we define two functors:
		
		\begin{enumerate}
			\item $\Phi_{\cD}\colon\cD\to\cD$ is the restriction of $\Phi_{\cC}$ on $\cD$, i.e.,  
			\[\Phi_{\cD} := \Phi_{\cC}|_{\cD} .\]
			
			\item $\Phi_{\cC/\cD}\colon\cC/\cD\to\cC/\cD$ is the unique (up to isomorphism) $A_{\infty}$-functor satisfying
			\begin{equation}\label{eq functor descended to quotient}
				\Phi_{\cC/\cD}\circ l\simeq l\circ \Phi_{\cC}
			\end{equation}
			where $l\colon\cC\to\cC/\cD$ is the localisation functor.
		\end{enumerate}
	\end{dfn}
	
	\begin{rmk}
		The existence and uniqueness of $\Phi_{\cC/\cD}$ follow from the universal property of $A_{\infty}$/dg localisation (see e.g. \cite{Toen}): Consider the $A_{\infty}$/dg functor $l\circ \Phi_{\cC}\colon\cC\to\cC/\cD$. Since $(l\circ\Phi_{\cC})(D)\simeq 0$ for all $D\in\cD$, by the universal property, there exists a unique $A_{\infty}$-functor $\Phi_{\cC/\cD}$ (not necessarily a dg functor) up to isomorphism satisfying (\ref{eq functor descended to quotient}).
	\end{rmk}
	
	%
	The following proposition about the split-generation of quotients will be useful when comparing categorical entropies in the next subsection.
	\begin{prop}\label{prop split-generators of quotient}
		Let $\cC$ be a pretriangulated $A_{\infty}$ (or dg) category, and $\cD$ be its full pretriangulated $A_{\infty}$ (or dg) subcategory. Assume that $G\oplus D$ split generates $\cC$, and $D$ split generates $\cD$ for some $G\in\cC$ and $D\in\cD$. Then $G$ split-generates $\cC/\cD$.
		
	\end{prop}
	
	\begin{proof}
		Let  $C\in\cC/\cD$. Then we have $C\in\cC$, as $\cC$ and $\cC/\cD$ have the same objects. Hence, there exists $C'\in\cC$ such that $C\oplus C'\in\Tw(\{G\oplus D\}_{\cC})$. Then, by applying the localisation functor $l$ to $C\oplus C'\in\cC$, we get
		\[l(C\oplus C')=C\oplus C'\in \Tw(\{l(G\oplus D)\}_{\cC/\cD})=\Tw(\{G\oplus D\}_{\cC/\cD})\simeq\Tw(\{G\}_{\cC/\cD})\]
		since $D\simeq 0$ in $\cC/\cD$, where the second $C$ and $C'$ are in $\cC/\cD$. This proves that $G$ split-generates $\cC/\cD$.
	\end{proof}
	
	The following lemma will be the main ingredient of Theorem \ref{thm:entropy-loc}.
	
	\begin{lem}\label{lem:g-preserving}
		
		Let $\cC$ be a pretriangulated $A_{\infty}$ (or dg) category, and $\cD$ be its full pretriangulated $A_{\infty}$ (or dg) subcategory.
		Assume that $G\oplus D$ split-generates $\cC$, and $D$ split-generates $\cD$ for some $G\in\cC$ and $D\in\cD$.
		Let $X\in \cC$, and $\tilde L\in\Tw(\{G\}_{\cC/\cD})$ be a twisted complex for $l(X)\oplus B$ for some $B\in\cC/\cD$, i.e.,
		\[\tilde L\simeq l(X)\oplus B\]
		where $l\colon\cC\to\cC/\cD$ is the localisation functor. Then there exist a twisted complex $\tilde K\in\Tw(\{G,D\}_{\cC})$ for $X\oplus A$ for some $A\in\cC$, i.e.,
		\[\tilde K\simeq X\oplus A\]
		such that the $G$-components of $\tilde K$ are the same as the $G$-components of $\tilde L$.
		In particular,  	
		\[\len_{t,G}\tilde K=\len_t\tilde L.\]
	\end{lem}
	
	\begin{proof}		
		Assume that $\cC$ and $\cD$ are dg categories for simplicity following Remark \ref{rmk:motivation}\eqref{rmk a-inf to dg}. Before we start the proof, note that for any $X,Y\in\cC$, we have the isomorphism by \cite{Verdier77, Verdier96} (which is restated in \cite{Drinfeld})
		\begin{equation}\label{eqn isomorphisms}
			l\colon\lim_{\stackrel{\xrightarrow{\hspace{5mm}}}{(Y\to Z)\in Q_{Y}}}\Hom^*_{\cC}(X, Z)\overset{\sim}{\longrightarrow} \Hom^*_{\cC/\cD}(l(X),l(Y)) ,
		\end{equation}
		where $Q_{Y}$ is the (filtered) category of morphisms $f\colon Y\to Z$ in $\cC$ such that $\Cone(f)$ is homotopy equivalent to an object in $\cD$.
		This shows, for any $\beta\in\Hom^*_{\cC/\cD}(l(X),l(Y))$, there exists $Z\in\cC$ such that
		\begin{enumerate}
			\item[(i)] $Z\simeq\Cone(E \to Y)$ for some $E\in\cD$, and
			\item[(ii)] there is a morphism $\alpha\in\Hom^*_{\cC}(X,Z)$ so that $l(\alpha)\simeq\beta$.
		\end{enumerate}
		
		With the above arguments, we can prove the lemma by induction on the number of components of $\tilde L\in\Tw(\{G\}_{\cC/\cD})$.
		
		For the base step, let us assume that $\tilde L$ has no components, i.e., $\tilde L=0$. Note that $\tilde L\simeq l(X)\oplus B\simeq l(X\oplus A_0)$ for some $A_0\in\cC$ since $l$ is essentially surjective. Then, $X\oplus A_0\in\cC$ is in the kernel of $l$.
		It means that $X\oplus A_0\oplus A_1$ is homotopy equivalent to an object in $\cD$ for some $A_1\in\cC$ (see e.g. \cite{Krause}).
		Since $D$ split-generates $\cD$, there exists $A_2\in\cD\subset\cC$ such that $X\oplus A_0\oplus A_1\oplus A_2$ is homotopy equivalent to an object $\tilde K\in\Tw(\{D\}_{\cC})$, which has no $G$-components.
		
		For the inductive step, let us assume the induction hypothesis, i.e., we assume that the lemma holds for any $X\in\cC$ and $\tilde L\in\Tw(\{G\}_{\cC/\cD})$ satisfying $\tilde L\simeq l(X)\oplus B$ for some $B\in\cC/\cD$, whenever $\tilde L$ has $n$ many $G$-components.
		
		Now we consider $X\in\cC$ and $\tilde L\in\Tw(\{G\}_{\cC/\cD})$ such that $\tilde L\simeq l(X)\oplus B$ for some $B\in\cC/\cD$, and $\tilde L$ has $(n+1)$ many $G$-components. 
		We note that, as mentioned in Remark \ref{rmk iterated cone}\eqref{item:iterated}, the twisted complex $\tilde L$ can be written as an iterated cone. 
		Thus, there is a twisted complex $\tilde T\in\Tw(\{G\}_{\cC/\cD})$ and $d\in\Z$ such that 
		\begin{equation}\label{eq:l-tilde}
			\tilde L = \Cone\left(\tilde T \xrightarrow{\beta} G[d]\right), 
		\end{equation}
		where $\beta\in\Hom_{\cC/\cD}^0\left(\tilde T,G[d]\right)$, and $\tilde T$ has $n$ many $G$-components.
		Since the localisation functor $l\colon \cC\to\cC/\cD$ is essentially surjective, there exists $W\in\cC$ such that $l(W)\simeq \tilde T$. Moreover, by the induction hypothesis, there exist $\tilde S\in\Tw(\{G,D\}_{\cC})$ and $A'\in\cC$ such that $\tilde S\simeq W\oplus A'$, and $\tilde S$ and $\tilde T$ have the same $G$-components.
		
		From the isomorphism \eqref{eqn isomorphisms}, one can show that there is an object $Z_1 \in \cC$ such that 
		\begin{enumerate}
			\item[(i)'] $Z_1 \simeq \Cone(E_1 \to G[d])$ for some $E_1 \in \cD$, and
			\item[(ii)'] there is a morphism $\alpha \in \Hom_{\cC}^0\left(W, Z_1\right)$ so that $l(\alpha) \simeq \beta$. 
		\end{enumerate}
		
		Now, we consider the cone of $\alpha : W\to Z_1$. Let 
		\[Y := \Cone\left(W \xrightarrow{\alpha} Z_1\right) \in \cC.\]
		Then, one can easily observe that $Y$ is a lift of $\tilde L$ with respect to the localisation functor $l$, i.e., $l(Y) \simeq\tilde L$. Hence,
		\[l(Y)\simeq l(X)\oplus B\simeq l(X\oplus A_0)\]
		for some $A_0\in\cC$, since $l$ is essentially surjective.
		This shows that there is a homotopy equivalence $\delta : l(X\oplus A_0) \to l(Y)$. 
		Consequently, from \eqref{eqn isomorphisms}, one obtains $Z_2 \in \cC$ such that 
		\begin{enumerate}
			\item[(i)''] $Z_2 \simeq \Cone(E_2 \to Y)$ for some $E_2 \in \cD$, and
			\item[(ii)''] there is a morphism $\gamma \in \Hom_{\cC}^0\left(X\oplus A_0, Z_2\right)$ so that $l(\gamma) \simeq \delta$.
		\end{enumerate}
		
		Let 
		\[E_3 := \Cone\left(X\oplus A_0 \xrightarrow{\gamma}Z_2\right)\in\cC.\]
		Then, since $l(\gamma) \simeq \delta$ is a homotopy equivalence, 
		\[l(E_3) = l\left(\Cone(\gamma)\right) \simeq \Cone(\delta)\simeq 0.\]
		In other words, $E_3$ is a direct summand of an object in $\cD$ (see e.g. \cite{Krause}), i.e., $E_3\oplus A''\in\cD$ for some $A''\in\cC$.
		
		Collecting all the above results, we get 
		\begin{align}
			\notag
			X\oplus A_0 &\simeq \Cone\left(Z_2[-1] \to E_3[-1]\right) \\
			\notag
			&\simeq \Cone\left(\Cone(E_2 \to Y)[-1] \to E_3[-1]\right) \\
			\notag	
			&\simeq \Cone\left(\Cone\left(E_2 \to \Cone(W \xrightarrow{\alpha} Z_1)\right)[-1] \to E_3[-1] \right) \\
			\label{eqn induction}
			&\simeq \Cone\left(\Cone\left(E_2 \to \Cone\left(W \xrightarrow{\alpha} \Cone(E_1 \to G[d])\right)\right)[-1] \to E_3[-1]\right) .
		\end{align}
		As mentioned in Remark \ref{rmk iterated cone}\eqref{item:iterated}, one can convert the iterated cone \eqref{eqn induction} to a twisted complex for $X\oplus A_0$
		\[X\oplus A_0\simeq \left[\, E_2[1]\to W[1]\to E_1[1]\to G[d] \to E_3[-1]\,\right] .\]
		Here, the notation for the twisted complex is as in Remark \ref{rmk iterated cone}\eqref{item:notation}.
		Note that $E_1,E_2,E_3\oplus A''\in\cD$, and since $D$ split-generates $\cD$, there exist $A_1,A_2,A_3\in\cD\subset\cC$ such that
		\[E_1\oplus A_1,\quad E_2\oplus A_2,\quad E_3\oplus A''\oplus A_3\quad \in\Tw(\{D\}_{\cC}) .\]
		Also, recall that $\tilde S\in\Tw(\{G,D\}_{\cC})$ satisfies $\tilde S\simeq W\oplus A'$ for some $A'\in\cC$, and $\tilde S$ has the same $G$-components as $\tilde T$. Set $A:= A_0\oplus A_1\oplus A_2\oplus A_3\oplus A'\oplus A''\in\cC$. Then, define $\tilde K\in\Tw(\{G,D\}_{\cC})$ by
		\begin{align*}
			X\oplus A&\simeq \left[ (E_2\oplus A_2)[1]\to (W\oplus A')[1]\to (E_1\oplus A_1)[1]\to G[d] \to (E_3\oplus A''\oplus A_3)[-1]\right]\\
			&\simeq \left[ (E_2\oplus A_2)[1]\to \tilde S[1]\to (E_1\oplus A_1)[1]\to G[d] \to (E_3\oplus A''\oplus A_3)[-1]\right]:=\tilde K .
		\end{align*}
		Hence, the $G$-components of $\tilde K$ consist of the $G$-components of $\tilde S[1]$ and $G[d]$.
		We also know that the $G$-components of $\tilde L$ consist of the $G$-components of $\tilde T[1]$ and $G[d]$ by \eqref{eq:l-tilde}. Since $G$-components of $\tilde S$ and $\tilde T$ are the same, this proves the induction step.
	\end{proof}

	\subsection{Comparison of categorical entropies of $\Phi_\cC, \Phi_\cD$, and $\Phi_{\cC/\cD}$.}
	\label{subsection Theorem entropy-loc}
	Now, we are ready to compare the categorical entropies of an $A_{\infty}$/dg functor $\Phi_\cC$ on a pretriangulated $A_{\infty}$/dg category $\cC$, its restriction $\Phi_\cD$ to a full pretriangulated $A_{\infty}$/dg subcategory $\cD$ of $\cC$, and its descent $\Phi_{\cC/\cD}$ to the quotient $\cC/\cD$.
	
	
	It is easy to prove the next proposition, which is the first result comparing the categorical entropies of $\Phi_\cC$ and $\Phi_{\cC/\cD}$.
	
	\begin{prop}\label{prp:entropy-loc}
		Let $\cC$ be a pretriangulated $A_{\infty}$ (or dg) category and $\cD$ be its full pretriangulated $A_{\infty}$ (or dg) subcategory.
		Let $\Phi_{\cC}\colon \cC\to \cC$ be an $A_{\infty}$ (or dg) functor satisfying
		$\Phi_{\cC}(\cD)\subset \cD$, and $\Phi_{\cC/\cD}\colon\cC/\cD\to\cC/\cD$ be the induced functor as in Definition \ref{dfn:loc-induced-functor}. 
		Then, for any $t\in\R$, we have
		\[h_t(\Phi_{\cC/\cD})\leq h_t(\Phi_{\cC}) .\]
	\end{prop}
	
	\begin{proof}	
		Assume that $\overline G:=G\oplus D$ split-generates $\cC$, and $D$ split-generates $\cD$ for some $G\in\cC$ and $D\in\cD$. Then by Proposition \ref{prop split-generators of quotient}, $G$ split-generates $\cC/\cD$.
		For a fixed $n \in \mathbb N$, let $\tilde K\in\Tw(\{G,D\}_{\cC})$ be a twisted complex for $\Phi^n_{\cC}(\overline G)\oplus C$ for some $C\in\cC$. Applying the localisation functor $l$ to $\tilde K$, and replacing $D$ with $0$, we get a twisted complex $\tilde L\in\Tw(\{G\}_{\cC/\cD})$ for
		\[l(\Phi^n_{\cC}(\overline G)\oplus C)\simeq\Phi^n_{\cC/\cD}(l(\overline G))\oplus l(C)\simeq\Phi^n_{\cC/\cD}(G)\oplus l(C) .\]
		Moreover, the $G$-components of $\tilde K$ are preserved in the new twisted complex $\tilde L$. 
		Thus,
		\[\len_t\tilde L=\len_{t,G} \tilde K\leq \len_t\tilde K.\]
		This shows that
		\[\delta_t(G;\Phi_{\cC/\cD}^n(G))\leq \delta_t(G,D;\Phi_{\cC}^n(\overline G)),\]
		which implies that
		\[h_t(\Phi_{\cC/\cD})\leq h_t(\Phi_{\cC})\]
		by Proposition \ref{prp:split-entropy}.
	\end{proof}

	Proposition \ref{prp:entropy-loc} gives a lower bound for the entropy of $\Phi_\cC$. The following, which is one of our main theorems, gives an upper bound.
	
	\begin{thm}\label{thm:entropy-loc}
		Let $\cC$ be a pretriangulated $A_{\infty}$ (or dg) category and $\cD$ be its full pretriangulated $A_{\infty}$ (or dg) subcategory.
		Let $\Phi_{\cC}\colon \cC\to \cC$ be an $A_{\infty}$ (or dg) functor satisfying
		$\Phi_{\cC}(\cD)\subset\cD,$
		and let
		$\Phi_{\cD}\colon \cD\to \cD$
		and
		$\Phi_{\cC/\cD}\colon\cC/\cD\to\cC/\cD$
		be the induced functors as in Definition \ref{dfn:loc-induced-functor}. Then, for any $t\in\R$, we have
		\[h_t(\Phi_{\cC/\cD})\leq h_t(\Phi_{\cC})\leq\max\{h_t(\Phi_{\cC/\cD}),h_t(\Phi_{\cD})\} .\]
	\end{thm}
	
	\begin{rmk}
	More generally, Theorem \ref{thm:entropy-loc} holds if we replace the pretriangulated $A_{\infty}$/dg categories $\cC$ and $\cD$ with triangulated categories, and the $A_{\infty}$/dg functor $\Phi_{\cC}$ with an exact functor. See Remark \ref{rmk:motivation}\eqref{rmk:triangulated-cat} for more details.
	\end{rmk}
	
	The idea of the proof of Theorem \ref{thm:entropy-loc} is roughly as follows: Given the setting of Proposition \ref{prop split-generators of quotient}, where $\cC/\cD$ is split-generated by $G\in\cC$, $\cD$ is split-generated by $D\in\cD$, and $\cC$ is split-generated by $\overline G:=G\oplus D$, the initial goal is to achieve the inequality
		\begin{equation}\label{eq:max-inequality}
			\delta_t(G,D;\Phi^n_{\cC}(\overline G))\leq \delta_t(G;\Phi^n_{\cC/\cD}(G)) + \delta_t(D;\Phi^n_{\cD}(D))
		\end{equation}
		for sufficiently large $n$, which would imply the inequality $h_t(\Phi_{\cC})\leq \max\{h_t(\Phi_{\cC/\cD}),h_t(\Phi_{\cD})\}$.
	
	An attempt would be to use Lemma \ref{lem:g-preserving}: For any twisted complex $\tilde L_n\in\Tw(\{G\}_{\cC/\cD})$ for $\Phi^n_{\cC/\cD}(G)\oplus B_n$ for some $B_n\in\cC/\cD$, there is a twisted complex $\tilde K_n\in\Tw(\{G,D\}_{\cC})$ for $\Phi^n_{\cC}(\overline G)\oplus A_n$ for some $A_n\in\cC$ such that $\len_{t,G}\tilde K_n=\len_t\tilde L_n$. Then we have
		\[\delta_t(G,D;\Phi^n_{\cC}(\overline G))\leq \len_t \tilde K_n=\len_{t,G}\tilde K_n+\len_{t,D}\tilde K_n=\len_t\tilde L_n+\len_{t,D}\tilde K_n .\]
		We can choose $\tilde L_n$ so that $\len_t\tilde L_n$ is close to $\delta_t(G;\Phi^n_{\cC/\cD}(G))$. Assuming that $\len_{t,D}\tilde K_n$ approaches $R_n\in[0,\infty]$ as $\len_t\tilde L_n$ approaches $\delta_t(G;\Phi^n_{\cC/\cD}(G))$, we get
		\[\delta_t(G,D;\Phi^n_{\cC}(\overline G))\leq \delta_t(G;\Phi^n_{\cC/\cD}(G)) + R_n .\]
		However, we have no control over $R_n$ via Lemma \ref{lem:g-preserving}. In particular, we cannot relate $R_n$ to $\delta_t(D;\Phi^n_{\cD}(D))$ using the lemma to achieve \eqref{eq:max-inequality}.
	
	The solution of this problem is the observation that the following also gives the entropy:
		\begin{equation}\label{eq:entropy-via-quotient}
			\frac{1}{n}\log\left(\frac{\delta_t(G,D;\Phi^{2n}_{\cC}(\overline G))}{\delta_t(G,D;\Phi^{n}_{\cC}(\overline G))}\right)\xrightarrow{n\to\infty}h_t(\Phi_{\cC}) .
		\end{equation}
		For that reason, we want to create a sequence of twisted complexes $\tilde K_i\in\Tw(\{G,D\}_{\cC})$ for $\Phi^{2^iN}_{\cC}(\overline G)\oplus A_i$ for some $A_i\in\cC$ such that approximately
		\begin{equation}\label{eq:quotient-inequality}
			\frac{\delta_t(G,D;\Phi^{2^{i+1}N}_{\cC}(\overline G))}{\delta_t(G,D;\Phi^{2^iN}_{\cC}(\overline G))}\leq\frac{\len_t\tilde K_{i+1}}{\len_t\tilde K_i}\leq\delta_t(G;\Phi^{2^iN}_{\cC/\cD}(G))+\delta_t(D;\Phi^{2^iN}_{\cD}(D))
		\end{equation}
		holds for a large $N$. Similar to \eqref{eq:max-inequality}, this would also imply Theorem \ref{thm:entropy-loc} by \eqref{eq:entropy-via-quotient}. The first inequality is approximately true for infinitely many $i$'s by a basic fact about sequences (given in the following Lemma \ref{lem:sequences}), since $\delta_t(G,D;\Phi^{2^iN}_{\cC}(\overline G))\leq\len_t\tilde K_i$.
	
	To achieve the second inequality, we create the sequence $\tilde K_i\in\Tw(\{G,D\}_{\cC})$ in a particular way: Define $\tilde K_0$ from $\tilde L_N\in\Tw(\{G\}_{\cC/\cD})$ by Lemma \ref{lem:g-preserving} as we mentioned above, so that $\len_{t,G}\tilde K_0=\len_t\tilde L_N$. Define $\tilde K_{i+1}$ from $\tilde K_i$ by applying $\Phi^{2^iN}_{\cC}$ to $\tilde K_i$, and then by replacing each $\Phi^{2^iN}_{\cC}(G)$ in $\tilde K_{i+1}$ by $\tilde K_i$ and each $\Phi^{2^iN}_{\cC}(D)=\Phi^{2^iN}_{\cD}(D)$ in $\tilde K_{i+1}$ by a twisted complex in $\{D\}_{\cD}$ whose length is close to $\delta_t(D;\Phi^{2^iN}_{\cD}(D))$. Note that $\len_{t,G}\tilde K_i=(\len_t\tilde L_N)^{2^i}$ holds. Hence, we approximately have
	\[\frac{\len_t\tilde K_{i+1}}{\len_t\tilde K_i}=\frac{\len_{t,G}\tilde K_i\cdot \len_t\tilde K_i +\len_{t,D}\tilde K_i\cdot \delta_t(D;\Phi^{2^iN}_{\cD}(D))}{\len_t\tilde K_i}\leq (\len_t\tilde L_N)^{2^i}+\delta_t(D;\Phi^{2^iN}_{\cD}(D)) .\]
	Since $(\len_t\tilde L_N)^{2^i}$ is approximately $\delta_t(G;\Phi^{2^iN}_{\cC/\cD}(G))$, this proves the second inequality in \eqref{eq:quotient-inequality}, and hence, Theorem \ref{thm:entropy-loc}. The rigorous proof of Theorem \ref{thm:entropy-loc} will appear after Lemma \ref{lem:sequences}.

	\begin{lem}\label{lem:sequences}
		
		Let $(a_i)_{i=0}^{\infty}$ and $(a'_i)_{i=0}^{\infty}$ be sequences of positive real numbers satisfying $a_i\leq a'_i$ for all $i$. Then we have
		\[\liminf_{i\to\infty}\frac{a_{i+1}/a_i}{a'_{i+1}/a'_i}\leq 1 .\]
	\end{lem}
	\begin{proof}
		Let us assume the contrary. 
		Then, there exist a real number $M>1$ and a natural number $N$ such that
		\[\frac{a_{i+1}/a_{i}}{a'_{i+1}/a'_{i}}\geq M\]
		for all $i\geq N$.
		However, for any $i>N$, we have
		\begin{gather*}
			\frac{a_N}{a'_N}\left(\prod_{j=N}^{i-1}\left(\frac{a_{j+1}/a_j}{a'_{j+1}/a'_j}\right)\right)
			= \frac{a_i}{a'_i}\leq 1
		\end{gather*}
		which gives a contradiction since
		\[\lim_{i\to\infty} \left(\prod_{j=N}^{i-1}\frac{a_{j+1}/a_j}{a'_{j+1}/a'_j}\right) = \infty .\]
	\end{proof}

	\begin{proof}[Proof of Theorem \ref{thm:entropy-loc}]
		
		Assume that $\overline G:=G\oplus D$ split-generates $\cC$, and $D$ split-generates $\cD$ for some $G\in\cC$ and $D\in\cD$. Then by Proposition \ref{prop split-generators of quotient}, $G$ split-generates $\cC/\cD$.
		The first inequality is Proposition \ref{prp:entropy-loc}, thus it is enough to prove the second inequality 
		\[h_t(\Phi_{\cC}) \leq \max\{h_t(\Phi_{\cC/\cD}), h_t(\Phi_\cD)\}.\]
		If $h_t(\Phi_\cC) = - \infty$, then the above holds.
		Thus, let us assume that $h_t(\Phi_\cC) > -\infty$. 
		
		We note that, in the rest of the proof, we use the following notation for convenience:
		\[\alpha := h_t(\Phi_\cC),\quad \beta := h_t(\Phi_{\cC/\cD}),\quad \gamma := h_t(\Phi_\cD).\]
		Moreover, we remark that we will be working with the twisted complexes $\tilde K_n\in\Tw(\{G,D\}_{\cC})$, $\tilde L_n\in\Tw(\{G\}_{\cC/\cD})$, and $\tilde M_n\in\Tw(\{D\}_{\cD})$ in the rest of the proof.
		\vskip0.1in 
		
		{\em Step 1}: We will relate the complexity $\delta_t(G,D;\Phi^n_{\cC}(\overline G))$ to $\alpha$.
		
		Let $\epsilon$ be a fixed positive real number. 
		Then, by Proposition \ref{prp:split-entropy}, there is a natural number $N_1$ such that for all $n \geq N_1$, the following holds:
		\begin{equation*}
			\left| \frac{1}{n} \log \delta_t(G,D; \Phi_\cC^n(\overline G)) - \alpha \right| < \epsilon.
		\end{equation*}
		
		Thus, for any $n\geq N_1$, the following holds:
		\begin{gather}
			\label{eqn for K_n} e^{n (\alpha - \epsilon)} < \delta_t(G,D; \Phi_\cC^n(\overline G)) < e^{n(\alpha + \epsilon)}.
		\end{gather}
		
		The second step is to choose twisted complexes $\tilde L_n\in\Tw(\{G\}_{\cC/\cD})$ and $\tilde M_n\in\Tw(\{D\}_{\cD})$, whose lengths are related to $\beta$ and $\gamma$ respectively.
		For the choices, we consider three different cases, the first case is $\beta > -\infty, \gamma > - \infty$, the second case is $\beta = -\infty, \gamma = -\infty$, and the third case is that one of $\{\beta, \gamma\}$ is $-\infty$ and the other is not.
		We will consider the first case in steps 2 -- 4, the second case in step 5, and the third case in step 6.
		\vskip0.1in 
		
		{\em Step 2}: As mentioned above, we assume that $\beta > - \infty, \gamma > - \infty$. 
		Then, for the fixed $\epsilon>0$ in step 1, there is a natural number $N_2$ (resp.\ $N_3$) such that 
		\begin{gather*}
			\left| \frac{1}{n} \log \delta_t(G; \Phi_{\cC/\cD}^n(G)) - \beta \right| < \epsilon, \text{  for all  } n \geq N_2, \text{  and  }
			\left| \frac{1}{n} \log \delta_t(D; \Phi_{\cD}^n(D)) - \gamma \right| < \epsilon, \text{  for all  } n \geq N_3.
		\end{gather*}
		
		Moreover, for any $n \geq N_2$ (resp.\ $n\geq N_3$), there is a twisted complex $\tilde L_n\in\Tw(\{G\}_{\cC/\cD})$ (resp.\ $\tilde M_n\in\Tw(\{D\}_{\cD})$) such that 
		\begin{itemize}
			\item $\tilde L_n$ is a twisted complex for $\Phi^n_{\cC/\cD}(G) \oplus B_n$ for some $B_n\in\cC/\cD$ (resp.\ $\tilde M_n$ is a twisted complex for $\Phi^n_{\cD}(D) \oplus C_n$ for some $C_n\in\cD$),
			\item the following hold:
			\begin{gather}
				\label{eqn for L_n} \len_t\tilde L_n < e^{n(\beta + \epsilon)}, \\
				\label{eqn for M_n} \len_t\tilde M_n < e^{n(\gamma + \epsilon)}.
			\end{gather}
		\end{itemize}
		\vskip0.1in
		
		{\em Step 3}: 
		We fix a sufficiently large integer $N$ such that 
		\[N \geq \max\{N_1, N_2, N_3\}.\]
		Then, we will construct a sequence of twisted complexes $\tilde K_i\in\Tw(\{G,D\}_{\cC})$ (which will be compared to $\delta_t(G,D; \Phi_\cC^{2^iN}(\overline G))$ as described right above of Lemma \ref{lem:sequences}) so that 
		\begin{enumerate}
			\item[(i)] $\tilde K_i$ is a twisted complex for $\Phi_\cC^{2^iN}(\overline G) \oplus A_i$ for some $A_i\in\cC$,
			\item[(ii)] the following equalities hold:
			\begin{align}
				\label{eqn condition for E_i}
				\len_{t,G}\tilde K_i &= (\len_t\tilde L_N)^{2^i} ,\\
				\label{eqn length of E_i}
				\len_t \tilde K_{i+1} &=\len_{t,G}\tilde K_i \cdot \len_t\tilde K_i + \len_{t,D}\tilde K_i \cdot \len_t\tilde M_{2^iN}.
			\end{align}
		\end{enumerate}
		
		We will choose $\tilde K_i$ inductively. 
		
		For the base case $i=0$, observe that
		\[\tilde L_N \simeq \Phi^N_{\cC/\cD}(G) \oplus B_N\simeq \Phi^N_{\cC/\cD}(l(\overline G)) \oplus B_N\simeq l(\Phi^N_{\cC}(\overline G))\oplus B_N .\]
		Then, by Lemma \ref{lem:g-preserving}, there exists a twisted complex $\tilde K_0\in\Tw(\{G,D\}_{\cC})$ such that
		\begin{itemize}
			\item $\tilde K_0$ is a twisted complex for $\Phi^N_\cC(\overline G) \oplus A_0$ for some $A_0\in\cC$,
			\item $\len_{t,G}\tilde K_0 = \len_t\tilde L_N$ .
		\end{itemize}
		We note that the last item is \eqref{eqn condition for E_i} for $i=0$.
		
		In order to choose $\tilde K_i$ for all $i \in \mathbb{N}$ inductively, let us assume that there is $\tilde K_i\in\Tw(\{G,D\}_{\cC})$ satisfying (i)--(ii) above.
		One can apply $\Phi_\cC^{2^iN}$ to the twisted complex $\tilde K_i$ to obtain a twisted complex
		\[\tilde K'_{i+1}\in\Tw(\{\Phi_\cC^{2^iN}(G),\Phi_\cC^{2^iN}(D)\}_{\cC})\]
		such that 
		\begin{itemize}
			\item $\tilde K_{i+1}'$ is a twisted complex for $\Phi_\cC^{2^iN}(\Phi_\cC^{2^iN}(\overline G) \oplus A_i) = \Phi_\cC^{2^{i+1}N}(\overline G) \oplus \Phi_\cC^{2^iN}(A_i)$,
			\item $\len_{t,\Phi_{\cC}^{2^iN}(G)}\tilde K_{i+1}'$ (resp.\ $\len_{t,\Phi_{\cC}^{2^iN}(D)}\tilde K_{i+1}'$) is equal to $\len_{t,G}\tilde K_i$ (resp.\ $\len_{t,D}\tilde K_i$).
		\end{itemize}
		
		We would like to construct a twisted complex $\tilde K_{i+1}\in\Tw(\{G,D\}_{\cC})$ by modifying $\tilde K_{i+1}'$ as follows:
		The components of $\tilde K_{i+1}'$ are shifts of $\Phi_\cC^{2^iN}(G)$ and $\Phi_\cC^{2^iN}(D)$.
		We replace each $\Phi_\cC^{2^iN}(D)$ in the components of $\tilde K_{i+1}'$ with
		\[\Phi_\cC^{2^iN}(D) \oplus C_{2^iN} = \Phi_\cD^{2^iN}(D) \oplus C_{2^iN}\]
		then replace it with the equivalent twisted complex $\tilde M_{2^iN}\in\Tw(\{D\}_{\cD})\subset\Tw(\{G,D\}_{\cC})$.
		Similarly, we replace each $\Phi_\cC^{2^iN}(G)$ in the components of $\tilde K_{i+1}'$ with 
		\[\Phi_\cC^{2^iN}(G) \oplus \Phi_\cC^{2^iN}(D) \oplus A_i = \Phi_\cC^{2^iN}(\overline G) \oplus A_i\]
		then replace it with the equivalent twisted complex $\tilde K_i\in\Tw(\{G,D\}_{\cC})$.
		
		Since the replacements can be understood as taking direct sums with $C_{2^iN}$'s and $\Phi_\cC^{2^iN}(D) \oplus A_i$'s, we get a twisted complex $\tilde K_{i+1}\in\Tw(\{G,D\}_{\cC})$ for $\Phi_\cC^{2^{i+1}N}(\overline G) \oplus A_{i+1}$ for some $A_{i+1}\in\cC$ after the replacements.
		Also, the following hold:
		\[\len_{t,G}\tilde K_{i+1}=\len_{t,\Phi_{\cC}^{2^iN}(G)}\tilde K_{i+1}'\cdot\len_{t,G}\tilde K_i = \len_{t,G}\tilde K_i\cdot \len_{t,G}\tilde K_i = (\len_t\tilde L_N)^{2^i} \cdot  (\len_t\tilde L_N)^{2^i}= (\len_t \tilde L_N)^{2^{i+1}},\]
		and
		\[\len_t \tilde K_{i+1} = \len_{t,\Phi_{\cC}^{2^iN}(G)}\tilde K_{i+1}'\cdot\len_t\tilde K_i + \len_{t,\Phi_{\cC}^{2^iN}(D)}\tilde K_{i+1}'\cdot \len_t\tilde M_{2^iN}= \len_{t,G}\tilde K_i \cdot \len_t\tilde K_i + \len_{t,D}\tilde K_i \cdot \len_t\tilde M_{2^iN}.\]
		Hence, (i)--(ii) holds for $\tilde K_{i+1}$. This concludes the construction of the sequence of twisted complexes $\tilde K_i\in\Tw(\{G,D\}_{\cC})$.
		\vskip0.1in
		
		{\em Step 4}: We will prove Theorem \ref{thm:entropy-loc} for the first case, i.e., the case of $\beta > -\infty, \gamma > -\infty$. 
		If we assume that $\beta \geq \gamma$, then we would like to prove that  
		\[ h_t(\Phi_\cC) = \alpha \leq \beta = \max\{h_t(\Phi_{\cC/\cD}) = \beta, h_t(\Phi_\cD) = \gamma\}.\]
		Based on this, we will find a contradiction under the assumptions that $\beta \geq \gamma$ and $\alpha > \beta$ using Lemma \ref{lem:sequences}.
		
		Consider the sequences
		\[a_i := \delta_t(G,D; \Phi_\cC^{2^iN}(\overline G)),\quad a_i' := \len_t\tilde K_i .\] 
		Obviously, we have $a_i\leq a'_i$ for all $i$ by the definition of the complexity. In order to study the sequence
		\[\frac{a_{i+1}/a_{i}}{a'_{i+1}/a'_{i}} ,\]
		we note that 
		\[e^{2^iN(\alpha - \epsilon)} < a_i < e^{2^iN(\alpha + \epsilon)},\quad \hspace{0.5em} e^{2^{i+1}N(\alpha - \epsilon)} < a_{i+1} < e^{2^{i+1}N(\alpha + \epsilon)},\]
		from \eqref{eqn for K_n}. 
		This induces that 
		\begin{gather}
			\label{eqn a_i fraction} e^{2^iN(\alpha-3\epsilon)}< \frac{a_{i+1}}{a_i} < e^{2^iN(\alpha+3\epsilon)}.
		\end{gather}
		
		We also note that for any $i$,
		\begin{gather}
			\label{eqn a'_i fraction}
			\begin{split}
				\frac{a'_{i+1}}{a'_i} &= \len_{t,G}\tilde K_i + \frac{\len_{t,D}\tilde K_i \cdot \len_t\tilde M_{2^{i}N}}{\len_t\tilde K_i} \\
				&\leq \len_t(\tilde L_N)^{2^i} + \len_t\tilde M_{2^iN} \\
				&< e^{2^iN(\beta + \epsilon)} + e^{2^iN(\gamma + \epsilon)} \\
				&=(1+e^{2^i N(\gamma-\beta)})e^{2^i N(\beta +\varepsilon)}\\
				&\leq 2 e^{2^iN(\beta + \epsilon)},
			\end{split}
		\end{gather}
		from \eqref{eqn for L_n}--\eqref{eqn length of E_i} and the assumption that $\beta \geq \gamma$.
		
		Since we assume that $\alpha > \beta$, by choosing a sufficiently small $\epsilon >0$ and a sufficiently large $N$, one can have 
		\[ N (\alpha -\beta - 4 \epsilon) > 1 .\]
		Then, \eqref{eqn a_i fraction} and \eqref{eqn a'_i fraction} show that
		\[\frac{a_{i+1}/a_{i}}{a'_{i+1}/a'_{i}} >\frac{1}{2}e^{2^{i}N(\alpha - 3 \epsilon) - 2^{i}N(\beta + \epsilon)}=\frac{1}{2}e^{2^{i}N(\alpha - \beta - 4\epsilon)}\]
		for all $i$. This implies
		\[\lim_{i\to\infty}\frac{a_{i+1}/a_{i}}{a'_{i+1}/a'_{i}}=\infty\]
		which contradicts with Lemma \ref{lem:sequences}. Thus, $\alpha \leq \beta$ if $\beta \geq \gamma$.
		\vskip0.1in
		
		%
		
		%
		
		If we assume that $\gamma > \beta$ instead of $\beta\geq\gamma$, then we should show that $\alpha \leq \gamma$. 
		Under the assumption that $\alpha > \gamma$, the same arguments give a contradiction again. 
		
		To sum up, steps 2--4 prove Theorem \ref{thm:entropy-loc} for the first case, i.e., $\beta>-\infty$ and $\gamma>-\infty$.
		\vskip0.1in
		
		{\em Step 5}: Here, we will consider the second case where
		\[\beta = \gamma = -\infty.\]
		We will show that $\alpha = -\infty$ by contradiction. 
		Since $\beta = h_t(\Phi_{\cC/\cD}) = -\infty$, for any $R \in \R$, there is a $N_2 \in \mathbb{N}$ such that for all $n \geq N_2$, 
		\[\frac{1}{n} \log \delta_t (G; \Phi^n_{\cC/\cD}(G)) < R.\]
		Furthermore, there is a twisted complex $\tilde L_n\in\Tw(\{G\}_{\cC/\cD})$ such that for any $\epsilon>0$,
		\begin{gather}
			\label{eqn new for L_n} \tag{\ref{eqn for L_n}$'$}
			\len_t\tilde L_n< e^{n(R + \epsilon)}.
		\end{gather}
		Similarly, there is $N_3 \in \mathbb{N}$ such that for all $n \geq N_3$, there exists a twisted complex $\tilde M_n\in\Tw(\{D\}_{\cD})$ so that for any $\epsilon>0$,
		\begin{gather}
			\label{eqn new for M_n} \tag{\ref{eqn for M_n}$'$}
			\len_t\tilde M_n< e^{n(R + \epsilon)}.
		\end{gather}	
		
		We can choose $R < \alpha$. 
		Then, by repeating the arguments in steps 2--4 with slight modifications, one can prove Theorem \ref{thm:entropy-loc} for the second case. 
		The modifications are using \eqref{eqn new for L_n} and \eqref{eqn new for M_n}, instead of \eqref{eqn for L_n} and \eqref{eqn for M_n}, and are replacing both of $\beta$ and $\gamma$ with $R$.
		\vskip0.1in
		
		{\em Step 6}: Here, we will consider the third case, i.e., only one of $\{\beta, \gamma\}$ is $-\infty$.
		For convenience, let us assume that $\beta = -\infty < \gamma$. 
		Then, the arguments in steps 2--4 will work after slight modification, as we did in step 5.
		The slight modifications are using \eqref{eqn new for L_n} instead of \eqref{eqn for L_n}, and replacing $\beta$ with a sufficiently small $R$. 
		
		When $\beta > -\infty = \gamma$, the same logic works. 
	\end{proof}
	
	We end this subsection by pointing out that under some assumptions, the categorical entropies of $\Phi_{\cC}$ and $\Phi_{\cC/\cD}$ agree with each other:
	\begin{cor}\label{cor:entropy-loc}
		Let $\cC$ be a pretriangulated $A_{\infty}$ (or dg) category and $\cD$ be its full pretriangulated $A_{\infty}$ (or dg) subcategory.
		Let $\Phi_{\cC}\colon \cC\to \cC$ be an $A_{\infty}$ (or dg) functor satisfying
		$\Phi_{\cC}(\cD)\subset\cD,$
		and let
		$\Phi_{\cD}\colon \cD\to \cD$
		and
		$\Phi_{\cC/\cD}\colon\cC/\cD\to\cC/\cD$
		be the induced functors as in Definition \ref{dfn:loc-induced-functor}.
		If $\Phi_{\cC/\cD}$ is a quasi-equivalence and $h_0(\Phi_{\cD})=0$, then
		\[h_0(\Phi_{\cC/\cD})=h_0(\Phi_{\cC}) .\]
	\end{cor}
%
	
	
	\subsection{Categorical entropy via admissible subcategories}
	\label{subsection when D is admissible}
	
	
	In this subsection, let $\cC$ be a pretriangulated $A_{\infty}$ (or dg) category, and $\cD$ be its full pretriangulated $A_{\infty}$ (or dg) subcategory. In order to make an inequality in Theorem \ref{thm:entropy-loc} an equality, we will assume that $\cD$ is {\em admissible}. 
	See Definition \ref{dfn:orthogonalcomplement} for the definition of admissible subcategory. 
	Our motivation of the assumption will be given in the next subsection. Note that the results in this subsection also hold for triangulated categories and exact functors between them, as pointed out in Remark \ref{rmk:motivation}\eqref{rmk:triangulated-cat}.
	
	\begin{dfn}\label{dfn:orthogonalcomplement}
		\mbox{}
		\begin{enumerate}
			\item The {\em right (resp.\ left) orthogonal complement $\mathcal{D}^{\perp}$} (resp. $^{\perp}\mathcal{D}$) of $\mathcal{D}$ is the full pretriangulated $A_{\infty}$ (or dg) subcategory of $\mathcal{C}$, consisting of objects $K \in \cC$ such that 
			\[ \Hom^*( L,K) = 0,\forall L \in \mathcal{D} \hspace{0.5em} (\mathrm{resp.\ } \Hom^*(K,L) =0, \forall L \in \mathcal{D} ). \]
			\item A full pretriangulated $A_{\infty}$ (or dg) subcategory $\mathcal{D}$ of a pretriangulated $A_{\infty}$ (or dg) category $\mathcal{C}$ is said to be {\em right-admissible (resp. left-admissible)} if for any $L\in \mathcal{C}$, there is a distinguished triangle
			\[ L' \rightarrow L \rightarrow L '' \rightarrow L'[1] \]
			for some $L' \in \mathcal{D}$ and $L'' \in \mathcal{D}^{\perp}$ (resp.\ for some $L' \in ^{\perp} \mathcal{D}$ and $L'' \in \mathcal{D}$) and such a triangle is unique up to unique isomorphism.
		\end{enumerate}
	\end{dfn}

	It is an easy consequence of Definition \ref{dfn:orthogonalcomplement} that, for any right-admissible (resp. left-admissible) subcategory $\mathcal{D}$ of $\mathcal{C}$, its right (resp.\ left) orthogonal complement $\mathcal{D}^{\perp}$ (resp.\ $^{\perp} \mathcal{D}$) is left-admissible (resp.\ right-admissible) and the left (resp.\ right) orthogonal complement $^{\perp} (\mathcal{D}^{\perp})$ (resp.\ $(^{\perp} \mathcal{D})^{\perp}$) of $\mathcal{D}^{\perp}$ (resp.\ $^{\perp} \mathcal{D}$) is $\mathcal{D}$.

	Let us consider the composition of the inclusion $ \mathcal{D}^{\perp} \hookrightarrow \mathcal{C}$ (resp.\ $^{\perp} \mathcal{D} \hookrightarrow \mathcal{C}$) and the localization functor $l: \mathcal{C} \to \mathcal{C}/\mathcal{D}$, which we will denote by $l_{\mathcal{D}^{\perp}}$ (resp.\ $l_{^{\perp} \mathcal{D}}$). It was remarked in \cite{Drinfeld} that a subcategory $\mathcal{D}$ of $\mathcal{C}$ is right-admissible (resp.\ left-admissible) if and only if $\mathcal{D}$ is split-closed and the functor $l_{\mathcal{D}^{\perp}} : \mathcal{D}^{\perp} \to \mathcal{C}/\mathcal{D}$ (resp.\ $l_{^{\perp} \mathcal{D}} : ^{\perp} \mathcal{D} \to \mathcal{C}/\mathcal{D}$) is an equivalence.
	
	Now assume that an $A_{\infty}$/dg subcategory $\mathcal{D}$ of an $A_{\infty}$/dg category $\mathcal{C}$ is right-admissible and hence that $\mathcal{D}^{\perp}$ is a left-admissible subcategory of $\mathcal{C}$. Then the above observations imply that the functor 
	\[l_\mathcal{D} : \mathcal{D} = ^{\perp} (\mathcal{D}^{\perp}) \to \mathcal{C}/(\mathcal{D}^{\perp})\]
	is an equivalence. Let us denote by 
	\[p : \mathcal{C} \to \mathcal{D} \hspace{0.5em} \mathrm{(resp.\  }q :\mathcal{C} \to \mathcal{D}^{\perp})\] 
	the right (resp.\ left) adjoint of the inclusion 
	\[i : \mathcal{D} \to \mathcal{C} \hspace{0.5em} \mathrm{  (resp.\  }j : \mathcal{D}^{\perp} \to \mathcal{C})\]
	whose existence can be proven using the definition of admissibility.
	We note that 
	\[p(L) \simeq L' \hspace{0.5em} \mathrm{  (resp.\  } q(L) \simeq L''[-1]),\]
	where $L, L', L''$ are the same as in Definition \ref{dfn:orthogonalcomplement}.

	Furthermore assume that $\Phi_{\mathcal{C}} : \mathcal{C} \to \mathcal{C}$ is an $A_{\infty}$/dg functor. Then we define its induced functor $\Phi_{\mathcal{D}^{\perp}} : \mathcal{D}^{\perp} \to \mathcal{D}^{\perp}$ by
	\[  \Phi_{\mathcal{D}^{\perp}}    = q \circ \Phi_{\mathcal{C}} \circ j.\] 
	Analogously, for a left-admissible $A_{\infty}$/dg subcategory $\mathcal{D}$ of $\mathcal{C}$, one can define the induced functor 
	$\Phi_{^{\perp}\mathcal{D}}: ^{\perp} \mathcal{D} \to ^{\perp} \mathcal{D}$.
	
	Then we are ready to prove the following corollary of Theorem \ref{thm:entropy-loc}.
	\begin{cor}\label{cor:orthogonal}
		Let $\mathcal{D}$ be a right-admissible (resp.\ left-admissible) $A_{\infty}$ (or dg) subcategory of $\mathcal{C}$.
		Let $\Phi_{\mathcal{C}} : \mathcal{C} \to \mathcal{C}$ be an $A_{\infty}$ (or dg) functor satisfying $\Phi_{\mathcal{C}} (\mathcal{D} )\subset \mathcal{D}$. Then, for any $t \in \R$, we have
		\begin{equation}\label{equality orthogonal}
			h_t(\Phi_ {\mathcal{C}})  = \mathrm{max} \{ h_t(\Phi_{\mathcal{D}}), h_t(\Phi_{\mathcal{D}^{\perp}}) \} \hspace{0.5em} (\mathrm{resp.\  }  h_t(\Phi_ {\mathcal{C}})  = \mathrm{max} \{ h_t(\Phi_{\mathcal{D}}), h_t(\Phi_{^{\perp} \mathcal{D}}) \}  ). 
		\end{equation} 
	\end{cor}
	\begin{proof}
		Let us consider the case that $\mathcal{D}$ is right-admissible only since the proof for the other case is similar.
		
		First we show that, for any $t\in \R$,
		\begin{equation}\label{inequality rightadmissible}
			h_t(\Phi_{\mathcal{C}}) \geq h_t ( \Phi_{\mathcal{D}}).
		\end{equation}
		
		For that purpose, take a split-generator $D$ for $\mathcal{D}$ and a split-generator $E$ for $\mathcal{D}^{\perp}$. Then
		$C := D \oplus E$ is a split-generator for $\mathcal{C}$.
		
		It is easy to see that for any positive integer $m$, we have
		\begin{equation}\label{inequality rightadmissible2}
			\delta_t \left( C; \Phi_{\mathcal{C}}^m(C) \right) \geq \delta_t \left( p(C); (p \circ \Phi_{\mathcal{C}}^m)(C) \right). 
		\end{equation}
		
		But, for the right hand side of \eqref{inequality rightadmissible2}, we have
		\[ p(C) \simeq D, \]
		and
		\begin{align*}
			(p \circ \Phi_{\mathcal{C}}^m) (C) &= p \left( (\Phi_{\mathcal{D}}^m) (D) \oplus (\Phi_{\mathcal{C}}^m(E)\right)\\
			&= \Phi_{\mathcal{D}}^m (D) \oplus (p \circ \Phi_{\mathcal{C}}^m)(E).
		\end{align*}
		
		Hence, we have, for any $t\in \R$,
		\begin{equation}\label{inequality rightadmissible3}
			\begin{split}
				\delta_t \left( p(C); (p \circ \Phi_{\mathcal{C}}^m)(C) \right) &= \delta_t \left( D; \Phi_{\mathcal{D}}^m (D) \oplus (p \circ \Phi_{\mathcal{C}}^m) (E)\right)\\
				&\geq \delta_t \left(D;\Phi_{\mathcal{D}}^m (D)\right).
			\end{split}
		\end{equation}
		
		The inequalities \eqref{inequality rightadmissible2} and \eqref{inequality rightadmissible3} prove the assertion \eqref{inequality rightadmissible}.

		On the other hand, we have 
		\begin{gather*}\label{equality localization}
			l(L) \simeq l\left( \Cone(L''[-1] \to L')\right) \simeq l(L''[-1]),\\
			(l \circ j \circ q)(L) \simeq l(L''[-1]),
		\end{gather*}
		where $L, L'$ and $L''$ are the same as in Definition \ref{dfn:orthogonalcomplement}.
		Thus, $(l \circ j \circ q)(L) \simeq l(L)$.
		
		Now we show that 
		\begin{equation}\label{equality rightorthogonal}
			h_t (\Phi_{\mathcal{D}^{\perp}}) = h_t (\Phi_{\mathcal{C}/\mathcal{D} }), \forall t\in \R.
		\end{equation} 
		For a proof, observe that the following diagram commutes.
		\[  \xymatrix{  \mathcal{D}^{\perp} \ar[r]^{l_{\mathcal{D}^{\perp}} } \ar[d]_{\Phi_{\mathcal{D}^{\perp}} }   &  \mathcal{C}/\mathcal{D} \ar[d]^{\Phi_{\mathcal{C}/\mathcal{D}} } \\     \mathcal{D}^{\perp} \ar[r]_{l_{\mathcal{D}^{\perp}}} & \mathcal{C}/\mathcal{D}.  }     \]
		
		Indeed, we have, for any $K\in \mathcal{D}^{\perp}$,
		\begin{align*}
			(\Phi_{\mathcal{C}/\mathcal{D}} \circ l_{\mathcal{D}^{\perp}}) (K) &= (\Phi_{\mathcal{C}/\mathcal{D}} \circ l \circ j)(K)  \\
			&\simeq (l\circ \Phi_{\mathcal{C}}\circ j) (K) \\
			&\simeq (l \circ j\circ q \circ \Phi_{\mathcal{C}} \circ j) (K)\\
			&= (l \circ j \circ \Phi_{\mathcal{D}^{\perp}}) (K) \\
			&=  (l_{\mathcal{D}^{\perp}} \circ \Phi_{\mathcal{D}^{\perp}}) (K).
		\end{align*}
		Here we used \eqref{equality localization} for the third equality.
		
		Hence, two functors $\Phi_{\mathcal{D}^{\perp}} :\mathcal{D}^{\perp} \to \mathcal{D}^{\perp}$ and $\Phi_{\mathcal{C}/\mathcal{D}} :  \mathcal{C}/\mathcal{D} \to \mathcal{C}/\mathcal{D}$ are identified on object level via the equivalence 
		\[l_{\mathcal{D}^{\perp}} : \mathcal{D}^{\perp} \to \mathcal{C}/\mathcal{D}.\] 
		Since the categorical entropy depends on how a functor acts on a generator iteratively,
		this proves the assertion.

		Now, Theorem \ref{thm:entropy-loc} and \eqref{equality rightorthogonal} say that, for any $t\in \R$, we have
		\begin{equation}\label{inequality rightorthogonal}
			h_t ( \Phi_{\mathcal{D}^{\perp}}) \leq h_t(\Phi_{\mathcal{C}}) \leq  \mathrm{max} \{ h_t(\Phi_{\mathcal{D}}),h_t(\Phi_{\mathcal{D}^{\perp}}) \}.
		\end{equation}

		The desired equality \eqref{equality orthogonal} follows from \eqref{inequality rightadmissible} and \eqref{inequality rightorthogonal}.
	\end{proof}

	\begin{rmk}
		Corollary \ref{cor:orthogonal} is also proved in a recent work \cite[Proposition 2.8]{Kim22}. 
		We note that Theorem \ref{thm:entropy-loc} can be applied to any full subcategory, but Corollary \ref{cor:orthogonal} cannot.
	\end{rmk}

	\subsection{Application: Categorical entropy on Fukaya categories}
	\label{section Entropy for wrapped Fukaya categories via compactly supported symplectomorphisms}
	
	We discuss one direct application of the previous results on Fukaya categories.
	Let $W$ be a Weinstein manifold of finite type and $\Lambda$ be a Legendrian stop in its ideal boundary $\partial_{\infty}W$.
	Write 
	\begin{itemize}
		\item $\cW(W)$ for the pretriangulated closure of its wrapped Fukaya category of $W$, which is an $A_{\infty}$-category generated by the Lagrangian cocores $G_1,\ldots,G_m$, 
		\item $\cW(W,\Lambda)$ for the pretriangulated closure of the partially wrapped Fukaya category of $(W,\Lambda)$.
	\end{itemize} 
	We assume $2 c_1(W)=0$ so that relevant categories are $\Z$-graded with respect to the choice of a complex volume form. 
	If $2 c_1(W)\neq 0$, the categories will be $\Z/2$-graded, in which case, we can only talk about the categorical entropy when $t=0$. 
	See Section \ref{subsubsection fukaya category} for a brief introduction to wrapped Fukaya categories of Weinstein manifolds. Define
	\[G:=G_1\oplus\ldots\oplus G_m .\]
	Let $D_1,\ldots,D_r$ be linking disks corresponding to $\Lambda$, and let $\cD$ be the full pretriangulated $A_{\infty}$-subcategory of $\cW(W,\Lambda)$ generated by $D_1,\ldots,D_r$. Define
	\[D:=D_1\oplus\ldots\oplus D_r\quad\text{and}\quad\overline G:=G\oplus D .\]
	By \cite{Ganatra-Pardon-Shende20}, we can write
	\[\cW(W)\simeq\cW(W,\Lambda)/\cD .\]
	Then there is the localisation functor
	\[l\colon\cW(W,\Lambda)\to\cW(W) .\]
	Note that $\overline G$ split-generates $\cW(W,\Lambda)$, $G$ split-generates $\cW(W)$, and $D$ split-generates $\cD$.
	
	Let $\phi\colon W\to W$ be a compactly supported exact symplectic automorphism, suitably graded so that it acts as an identity on $D_i$'s. See Subsection \ref{subsec: grading of symplectic automorphism} for a detailed discussion. Then there are induced functors
	\[\Phi_{\cW(W,\Lambda)}\colon\cW(W,\Lambda)\to\cW(W,\Lambda),\quad\text{and}\quad\Phi_{\cW(W)}\colon\cW(W)\to\cW(W) ,\]
	satisfying
	\[\Phi_{\cW(W)}\circ l\simeq l\circ\Phi_{\cW(W,\Lambda)} .\]
	
	\begin{prop}\label{prp:compactly supported}
		Let $\phi\colon W\to W$ be a compactly supported exact symplectic automorphism. 
		Then, the induced functor $\Phi_{\cD}\colon \cD\to\cD$ satisfies
		\[h_t(\Phi_{\cD}) \leq 0 .\]
	\end{prop}
	\begin{proof}
		Since $\phi(D_i)=D_i$ for any $i$, $\Phi_\cD$ is the identity functor. The result follows.
	\end{proof}
	
	\begin{thm}
		\label{thm:entropy-w-pw}
		Let $\phi\colon W\to W$ be a compactly supported exact symplectic automorphism, and $\Lambda\subset\partial_{\infty}W$ is a stop. 
		\begin{enumerate}
			\item If we have $h_t(\Phi_{\cW(W)})\geq 0$ for a given $t\in\R$, then
			\begin{equation*}
				h_t(\Phi_{\cW(W)})=h_t({\Phi_{\cW(W,\Lambda)}}).
			\end{equation*}
			\item If, moreover, we have $h_{-t}(\Phi_{\cW(W)}^{-1}) \geq 0$, then
			\begin{equation*}
				h_t(\Phi_{\cW(W)})=h_t({\Phi_{\cW(W,\Lambda)}})=h_{-t}({\Phi^{-1}_{\cW(W,\Lambda)}})=h_{-t}(\Phi^{-1}_{\cW(W)}).
			\end{equation*}
			\item In particular, the above assumptions hold for $t=0$ and we have
			\begin{equation*}
				h_0(\Phi_{\cW(W)})=h_0({\Phi_{\cW(W,\Lambda)}})=h_0({\Phi^{-1}_{\cW(W,\Lambda)}})=h_0(\Phi^{-1}_{\cW(W)}).
			\end{equation*}
		\end{enumerate}
	\end{thm}
	\begin{proof}
		The first item (1) is a direct corollary of Theorem \ref{thm:entropy-loc}.
		Take $\cC:= \cW(W,\Lambda)$, and let $\cD$ be the category defined as above. Then, Proposition \ref{prp:compactly supported} and Theorem \ref{thm:entropy-loc} complete the proof of (1).
		
		If $h_t(\Phi_{\cW(W)})\geq 0$ and if $h_{-t}(\Phi_{\cW(W)}^{-1}) \geq 0$, we have $h_t(\Phi_{\cW(W)})=h_t({\Phi_{\cW(W,\Lambda)}})$ and $h_{-t}({\Phi^{-1}_{\cW(W,\Lambda)}})=h_{-t}(\Phi^{-1}_{\cW(W)})$ because of (1).
		Thus, in order to prove (2), it is enough to prove $h_t({\Phi_{\cW(W)}})=h_{-t}({\Phi^{-1}_{\cW(W)}})$.
		
		For that, we fix a Lefschetz fibration $\pi: W \to \mathbb{C}$.
		We note that the existence of $\pi$ is proven in \cite{Giroux-Pardon17}.
		And, we set $\Lambda_0 := \pi^{-1}(-\infty)$. 
		Then, as mentioned in \cite{Ganatra-Pardon-Shende18b}, $\cW(W,\Lambda_0)$ is a smooth and proper category.
		Thus, we have 
		\[h_t(\Phi_{\cW(W)})=h_t({\Phi_{\cW(W,\Lambda_0)}})=h_{-t}({\Phi^{-1}_{\cW(W,\Lambda_0)}})=h_{-t}(\Phi^{-1}_{\cW(W)}),\]
		by Proposition \ref{prp:op and inv}.
		
		Assumptions of (1) and (2) holds for $t=0$ because $\Phi_{\cW(W)}$ is a quasi-equivalence. Hence, (3) holds.
	\end{proof}
		
	\section{Entropies of exact symplectic automorphisms}
	\label{section entropies of products of Dehn twists}
	The results in Section \ref{section Entropy for wrapped Fukaya categories via compactly supported symplectomorphisms} compares categorical entropy on wrapped and partially wrapped Fukaya categories, and one of the main goals of Section \ref{section entropies of products of Dehn twists} is to compare those on compact and wrapped Fukaya categories. 
	To do that, as mentioned in Section \ref{section introduction}, we restrict our attention to symplectic manifolds satisfying a duality condition (between compact and wrapped Fukaya categories).
	Moreover, the comparison gives us a way to compute categorical entropy as an exponential growth of morphisms spaces. 
		
	Section \ref{section entropies of products of Dehn twists} is organized as follows:
	Section \ref{subsection setup} briefly reviews preliminaries.
	Also, we fix our Weinstein manifolds in Section \ref{subsection setup}, which are plumbing spaces of $T^*S^n$.  
	In Section \ref{subsection entropies of products of Dehn twists}, we prove Theorem \ref{thm:entropy}, i.e., the categorical entropy of any compactly supported symplectic automorphism can be computed in terms of Lagrangian Floer (co)homology.
	And this leads us to the proof of Theorem \ref{thm:equality}, studying relationships between categorical entropies on the compact and wrapped Fukaya categories.
	In Section \ref{subsection equality between categorical entropies}, we state Theorem \ref{thm general} combining all results proven in Section \ref{section entropies of products of Dehn twists}. 
	
	We note that a well-known example of Weinstein manifolds satisfying our duality condition is the plumbings of $T^*S^n$ along tree. 
	For convenience, we work on the plumbing spaces, but every argument in Section \ref{section entropies of products of Dehn twists} can be applied to any Weinstein manifold with the duality condition, without a modification. 
	\subsection{Setup}
	\label{subsection setup}
	
	\subsubsection{Preliminaries on compact and wrapped Fukaya categories}
	\label{subsubsection fukaya category}
		We briefly introduce the notion of compact Fukaya category and wrapped Fukaya category of a Weinstein manifold. See \cite{Seidel08,Ganatra-Pardon-Shende20} for more details.
		Let $(W,d\lambda)$ be a Weinstein manifold of finite type such that $2c_1(W)=0$. Let us fix a quadratic complex volume form $\eta^2$ on $W$. Then the Lagrangian Grassmannian bundle $\mathrm{Gr}(TW) \to W$, whose fiber at $p \in W$ is the Lagrangian Grassmannian $\mathrm{Gr}(T_p W)$, admits a squared phase map $\alpha : \Gr(TW)\to S^1.$

		Every Lagrangian submanifold $L$ of $W$ has a natural section $\beta_L : L \to \Gr(TW)$ mapping $p\in L$ to $T_p L \in \Gr(TW)$. We say that a pair $(L,\widetilde{\gamma}_L)$(or $L$ itself) is a graded Lagrangian if $L$ is a Lagrangian and $\widetilde{\gamma}_L : L \to \R$ is a lifting of $\gamma_L := \alpha \circ \beta_L: L \to S^1 = \R/\Z$.


		To deal with pin structures, we fix $b \in H^2(W,\Z/2)$ called a {\em background class}. We say that $L$ is relatively pin with respect to $b$ if the second Stiefel-Whitney class $w_2(L)$ coincides with $b|_L$. We will specify a particular choice of $b$ whenever it is necessary.

		\begin{dfn}
			\label{dfn fukaya categories}
			\mbox{}
			\begin{enumerate}
				\item Let {\em $\W(W)$} be the wrapped Fukaya category of $W$ whose objects are exact Lagrangian submanifolds $L$ of $W$ with a grading and a relatively pin structure such that $L$ is tangent to the Liouville vector field outside a compact subset of $W$,
				Let us call such a Lagrangian submanifold $L$ {\em admissible}.
				\item Let $\F(W)$ be the full subcategory of $\W(W)$ consisting of closed Lagrangians.
			\end{enumerate}
			By abuse of notation, let us denote $\Tw \W(W)$ and $\Tw \F(W)$ by $\W(W)$ and $\F(W)$, respectively. 			\end{dfn}

	\subsubsection{Plumbing spaces $P_n(T)$}
	Now, we will fix Weinstein manifolds that we will work with in the later parts of the paper. 
	
	Let $T$ be a tree and let $V(T)$ denote the set of vertices of $T$. 
	First, we introduce the following:
	\begin{dfn}
		\label{def P_n(T)}
		For a tree $T$, let {\em $P_n(T)$} denote the Weinstein manifold obtained by plumbing multiple copies of $T^*S^n$ along a tree $T$. See, e.g.,  \cite{Etgu-Lekili17} for a more detailed definition. 
	\end{dfn}
	One can set the following notation:
	\begin{dfn}
		\label{def spheres, dehn twists, etc}
		\mbox{}
		\begin{enumerate}
			\item By definition, for any vertex $v \in V(T)$, there is a Lagrangian sphere corresponding to $v$. 
			Let {\em $S_v$} denote the Lagrangian sphere corresponding to $v$.
			\item Let {\em $L_v$} denote a Lagrangian cocore of $S_v$. 
			In particular, $L_v$ is a Lagrangian disk such that 
			\begin{gather*}
				|L_v \cap S_w| = 
				\begin{cases}
					1, \text{  if  } v = w, \\
					0, \text{  otherwise.}
				\end{cases}		
			\end{gather*} 
		\end{enumerate}
	\end{dfn}
	
	\begin{rmk}
		\label{rmk cocore}
		\mbox{}
		\begin{enumerate}
			\item[(i)] To be more precise, the choice of "{\em Lagrangian cocore}" in Definition \ref{def spheres, dehn twists, etc} is not unique. 
			In fact, when $W$ is a Weinstein manifold, for any smooth point $p$ of a Lagrangian skeleton of $W$, there is a Lagrangian disk which transversally intersects the Lagrangian skeleton at $p$. 
			This is because, after a proper modification of the Liouville structure, the new Liouville structure near the smooth point $p$ is exactly the same {as} the standard Liouville structure of a cotangent bundle. 
			We note that the zero section of the cotangent bundle is a small neighborhood of $p$ in the Lagrangian skeleton. 
			For more details, we refer the reader to \cite[Section 9.1]{Chantraine-Rizell-Ghiggini-Golovko}. Also, see \cite[Sections 1.1 and 7.1]{Ganatra-Pardon-Shende18}.
			\item[(ii)] Also, for a tree $T$, $P_n(T)$ admits a Weinstein structure whose Lagrangian skeleton is the union of the zero sections, i.e., 
			\[\cup_{v \in V(T)}S_v.\]
			Thus, based on (i), for any point $p \in S_v$ except the plumbing points, there is a Lagrangian cocore $L_p$.
			Moreover, if $p_1, p_2 \in S_v$, then it is easy to check that $L_{p_1}$ and $L_{p_2}$ are Hamiltonian isotopic. 
			Thus, the Lagrangian cocore $L_v$ in Definition \ref{def spheres, dehn twists, etc} is well defined up to Hamiltonian isotopy.
		\end{enumerate}
	\end{rmk}

		For a given tree $T$, since $c_1(P_n(T))=0$ (see \cite{Etgu-Lekili17} for $n=2$), $P_n(T)$ has quadratic complex volume forms. Furthermore, if $n\geq 2$, then since $H^1(P_n(T),\Z) =0$, any two quadratic complex volume forms define equivalent gradings on Fukaya categories. If $n\geq3$, then $H^2(P_n(T), \mathbb Z/2)=0$ and the background class $b$ is zero. From now on, we will work with $n\geq 3$. Then, for every $v\in V(T)$, both $S_v$ and $L_v$ have a unique pin structure and they are gradable. Their gradings will be explained in Lemma \ref{lemma grading}.



	As mentioned before, it was shown in \cite{Ekholm-Lekili} that $\W (P_n(T))$ and $\F (P_n(T))$ are Koszul dual to each other for $n \geq 3$.
	Moreover, \cite[Section 4.2]{Abouzaid-Smith} shows the following lemma.
	\begin{lem}
		\label{lemma grading}
		If two vertices $v, w \in V(T)$ are connected by an edge in $T$, we write $v \sim w$. 
		Then, there are gradings for $S_v$ and $L_v$ for all $v\in V(T)$ and Floer data for $\W\left(P_n(T)\right)$ so that the following hold: 
		\begin{enumerate}
			\item For each $v \in V(T)$, 
			\[\hom^*(L_v, S_v) = \begin{cases} k \langle \overline{p}_v \rangle \simeq k& \text{if  }*=0, \\ 0 &\text{otherwise},\end{cases}\]
			or equivalently, 
			\[\hom^*(S_v, L_v) = \begin{cases} k \langle p_v \rangle \simeq k & \text{if  }*=n, \\ 0 &\text{otherwise},\end{cases}\]
			where $\{p_v\} = S_v \cap L_v$ and $\overline{p}_v$ denotes the same intersection point regarded as a morphism from $L_v$ to $S_v$.
			\item For any $v, w \in V(T)$ such that $v \neq w$, 
			\[\hom^*(S_v,L_w) = \hom^*(L_v, S_w) =0.\]
			\item For any $v \in V(T)$, 
			\[\hom^*(S_v,S_v) = \begin{cases} k & \text{if  } * = 0, n, \\ 0 &\text{otherwise}.\end{cases}\]
			\item For any $v, w \in V(T)$ such that $v \sim w$, there exists an integer $s_{vw} \in \{1, \dots, n-1\}$ such that
			\[\hom^{*} (S_v,S_w)  =\begin{cases} k & *= s_{vw},\\ 0 & \text{otherwise},\end{cases} \]
			and 
			\[s_{vw} + s_{wv} = n.\]
			\item For any $v, w \in V(T)$ such that $v \neq w, v \nsim w$, 
			\[\hom^{*} (S_v, S_w) = 0.\]
			\item For any $v,w\in V(T)$, $\Hom^*(S_v,S_w)$ is non-negatively graded and $\dim \Hom^0(S_v,S_v) =1$.
			\item For any $v,w \in V(T)$, $\Hom^*(L_v,L_w)$ is non-positively graded and $\dim \Hom^0(L_v,L_v)=1$.
		
		\end{enumerate}
	\end{lem} 
	In the rest of our paper, our morphism spaces of the wrapped Fukaya categories of plumbing spaces will be as in Lemma \ref{lemma grading}.
	
	\begin{rmk}
		We recall the reason why we care those specific spaces given in Section \ref{subsection setup}. If $W$ is a plumbing space of $T^*S^n$ with $n \geq 3$ along a tree, then there are many known facts about relations between the wrapped and compact Fukaya category of $W$, for example, the ``{\em Koszul duality}'' mentioned in \cite{Ekholm-Lekili}.
		Thus, the plumbing spaces can be good starting points to study the connection between entropies on different Fukaya categories of the same Weinstein manifold.
	\end{rmk}

	\subsubsection{Symplectic automorphism}
	\label{subsec: grading of symplectic automorphism}
	In this subsection, we discuss auto-equivalences on the compact and wrapped Fukaya categories induced by a symplectic automorphism. 

	Let $(W,d\lambda)$ be a Weinstein manifold of finite type such that $2c_1(W)=0$. Let a quadratic complex volume form $\eta^2$ on $W$, and let $b\in H^2(W,\Z/2)$ be given so that the wrapped Fukaya category $\W(W)$ and the compact Fukaya category $\F(W)$ make sense.
		
	For an exact symplectic automorphism $\phi$ on $W$ to induce auto-equivalences on $\W(W)$ and $\F(W)$,  $\phi$ is required to satisfy the following three conditions:
		\begin{enumerate}
			\item $\phi^*\lambda -\lambda =dh$ for some compactly supported function $h:W \to \R$.
			\item $\phi^* \eta^2 = f \eta^2$ for some non-vanishing complex-valued function $f: W \to \C^*$ which admits a lifting $\widetilde{f} = (\widetilde{f}_1,\widetilde{f}_2): W \to \R \times \R$ where $\R \times \R \to \C^*, (s,t) \mapsto e^{2\pi(s+it)}$ is regarded as a universal covering of $\C^*$.
			\item $ \phi^* b= b$.
		\end{enumerate}

	In particular, we use the following terminology used in \cite{Seidel06} for exact symplectic automorphisms satisfying the condition (1) above.
			\begin{dfn}\label{dfn liouville isomorphism}
	A Liouville automorphism of $(W,d\lambda)$ is an exact symplectic automorphism $\phi$ on $(W,d\lambda)$ satisfying the first condition (1) above.
	\end{dfn}
	Note that any compactly supported exact symplectic automorphism of $W$ is a Liouville automorphism of $W$.
	
	Now let $L$ be a Lagrangian submanifold of $W$ tangent to the Liouville vector field outside a compact subset of $W$. This is equivalent to saying that $\lambda|_{L}$ vanishes outside the compact subset. Therefore, the first condition (1) means that $\phi(L)$ is tangent to the Liouville vector field outside a compact subset of $W$ as well.
	
	The second condition (2) implies that $\phi$ defines a graded automorphism $(\phi,\gamma^\#)$ in the sense of \cite[Chapter (12i)]{Seidel08}. Indeed, $\gamma^\# :\Gr(TW) \to \R$ can be given by $\widetilde{f}_2 \circ \mathcal{P}$
	where $\mathcal{P} : \Gr(TW) \to W$ is the projection.
	This allows us to send a graded Lagrangian $(L,\widetilde{\gamma}_L)$ to another graded Lagrangian $(\phi(L), \widetilde{\gamma}_L \circ \phi^{-1}+ \gamma^\#\circ \beta_L \circ \phi^{-1})$.

	Furthermore, the third condition (3) ensures that for a Lagrangian $L$ relatively pin with respect to $b$, $\phi(L)$ is again relatively pin with respect to $b$ since $b|_{\phi(L)} = (\phi^{-1})^* b|_{\phi(L)} = (\phi^{-1})^* (b|_{L})= (\phi^{-1})^*w_2(L) = w_2(\phi(L))$. Furthermore, the pull-back $(\phi^{-1})^*$ sends a relative pin structure on $L$ to one on $\phi(L)$.

	In summary, if the conditions (1), (2) and (3) are satisfied, then $\phi$ induces well-defined auto-equivalences on both $\W(W)$ and $\F(W)$ by sending an admissible Lagrangian $L$ to another admissible Lagrangian $\phi(L)$ with a grading and a relative pin structure described above.
	Passing these to the pretriangulated closure, we define the following notations.
	\begin{dfn}
		Let $\Phi_{\F(W)}$ and $\Phi_{\W(W)}$ be the auto-equivalence induced by $\phi$ on $\F(W)$ and $\W(W)$, respectively.
	\end{dfn}

	Throughout the paper, for any Weinstein domain having exactly one boundary component(, which holds when its dimension is greater than or equal to $4$), and any compactly supported Liouville automorphism $\phi$ on its completion $W$, we will choose a graded automorphism $(\phi,\gamma^\#)$ in a specific way once the second condition (2) is satisfied. Indeed, we consider the graded automorphism $(\phi,\gamma^\#)$ such that $\gamma^\# = 0$ on $\mathcal{P}^{-1} (W \setminus \mathrm{Supp } (\phi) ) \subset \Gr(TW)$, which makes sense since $\mathrm{Supp } (\phi)$ is assumed to be compact and $\phi$ is the identity on the complement $W \setminus \mathrm{Supp } (\phi) $.
		
	Note that the second condition (2) holds for any $\phi$ if $H^1(W,\Z) =0$ and the third condition (3) holds for any $\phi$ if $b$ is chosen to be zero. In particular, in the case of plumbings, observe that $P_n(T)$ satisfies $H^1(P_n(T),\Z) =0$ for any $n\geq 2$ and any tree $T$. Furthermore, since $H^2(P_n(T), \Z/2) =0$ for any $n\geq 3$ and any tree $T$ as mentioned above, the background class $b$ is necessarily $0$ for the case $n\geq 3$. Therefore, any Liouville automorphism of $P_n(T)$ defines equivalences on Fukaya categories of $P_n(T)$.

	\subsection{Entropies of Liouville automorphisms}
	\label{subsection entropies of products of Dehn twists}
	Let $n\geq 3$. We note that the generation results for both $\F(P_n(T))$ and $\W(P_n(T))$ are well-known.
	For $\F(P_n(T))$, see \cite[Theorem 1.1]{Abouzaid-Smith}. 
	More precisely, it was shown that $\mathcal{F}(P_n(T))$ is generated by the spheres $\{S_v | v \in V(T)\}$.
	Hence we can take $S = \bigoplus_{v \in V(T)} S_v$ as a split-generator of $\mathcal{F}(P_n(T))$. 
	On the other hand, $\W(P_n(T))$ is generated by the cocores $\{L_{v}|  v\in V(T)\}$, and so we can take $L = \bigoplus_{v\in V(T)} L_v$ as a split-generator of $\W(P_n(T))$.
	
	As seen above, the endomorphism algebra $\hom^*(S,S) = \bigoplus_{v,w} \hom^*(S_v,S_w)$ is non-negatively graded and furthermore $\hom^0(S_v,S_v) = k \langle e_{S_v} \rangle$ for a cohomological unit $e_{S_v}$, which we will call the identity morphism of $S_v$. On the other hand, the endomorphism algebra $\Hom^* (L,L) = \bigoplus_{v,w}\Hom^* ( L_v, L_w)$ is known to be non-positively graded for certain gradings of $L_v$, and $\Hom^0(L_v,L_v)$ is generated by the strict unit. By using a minimal model for $\hom^*(L,L)$ (See \cite[Section (1i)]{Seidel08}), we may assume that $\hom^*(L_v,L_v)$ is also non-positively graded. This means that every morphism of $\hom^0(L_v,L_v)$ is a cocycle and is a multiple of a cohomological unit $e_{L_v} \in \hom^0(L_v,L_v)$, which we will call the identity morphism of $L_v$. 
	Note that $\{L_v|v\in V(T)\}$ still generates $\W(P_n(T))$ even after $\bigoplus_{v,w} \hom^*(L_i,L_j)$ is perturbed into a minimal model.

		Before stating Theorem \ref{thm:simpletwistedcomplex}, observe that Definition \ref{dfn:length}, Equation \eqref{eq:hom-space-tw}, and Lemma \ref{lemma grading} imply the following:
		\begin{lem}\label{lem:length1}
			\mbox{}
			\begin{enumerate}
				\item Let $\mathcal{S} \in \Tw \{S_v| v \in V(T)\}$. For any $v\in V(T)$ and $t\in \R$, the following equality holds:
				$$\len_{t,S_v} \mathcal{S} = \sum_{k \in \Z} \dim \hom^{n+k} \left(\mathcal{S}, L_v\right) e^{kt}.$$
				\item Let $\mathcal{L} \in \Tw \{L_v| v \in V(T)\}$. For any $v\in V(T)$ and $t\in \R$, the following equality holds:
				$$\len_{t,L_v} \mathcal{L} = \sum_{k \in \Z} \dim \hom^{n-k} \left(S_v,\mathcal{L}\right)e^{kt}.$$
			\end{enumerate}
		\end{lem}
	
	\begin{thm}\label{thm:simpletwistedcomplex}
		\mbox{}
		\begin{enumerate}
			\item (\cite[Lemma 2.5]{Abouzaid-Smith})
			For any admissible closed Lagrangian $E$ of $P_n(T)$, there is a twisted complex $\mathcal{E} \in \Tw \{S_v| v\in V(T)\}$ for $E$, in which none of the arrows are multiples of the identity morphisms.
			\item
			{For any admissible Lagrangian $E$ of $P_n(T)$, there is a twisted complex $\mathcal{E} \in \Tw \{L_v| v\in V(T)\}$ for $E$, in which none of the arrows are multiples of the identity morphisms.}
		\end{enumerate}
	\end{thm}
	\begin{proof}
	We refer to \cite[Lemma 2.5]{Abouzaid-Smith} for a proof of the first statement (1). 
	We remark that the statement of \cite[Lemma 2.5]{Abouzaid-Smith} is for plumbings of two cotangent bundles, but its proof works for arbitrary plumbings over a tree. See also \cite[Remark 1.2, Leamma 4.7]{Abouzaid-Smith}.
				
			Let us now prove the second statement (2). As mentioned above, any Lagrangian $E$ is quasi-isomorphic to a twisted complex $\mathcal{F} \in \Tw \{L_v | v\in V(T) \}$. It suffices to show that any twisted complex ${\mathcal{F}} \in \Tw \{ L_v | v \in V(T)\}$ is quasi-isomorphic to a twisted complex $\mathcal{E} \in \Tw \{ L_v | v \in V(T)\}$, in which none of the arrows are nonzero multiples the identity morphisms.
			
			Let ${\mathcal{F}} = \left[ (L_{v_i}[d_i])_{i=1}^{k}, (f_{ij}) \right]$ be given for some $k\in \mathbb{Z}_{\geq 1}$, $v_i \in V(T)$,  $d_i \in \mathbb{Z}$ and $f_{ij} \in \hom^1(L_{v_j}[d_j],L_{v_i}[d_i])$ for $i>j$.
			Since $\hom^*(L_{v_j},L_{v_i})$ is non-positively graded, if $f_{ij}$ is nonzero then, we have 
			\begin{equation}
				\label{eq:degree}
				d_i \leq d_j -1.
			\end{equation}
			Hence we may assume that the components $L_{v_1}[d_1],\dots, L_{v_k}[d_k]$ of $\mathcal{F}$ are ordered in such a way that
			\begin{equation}
				\label{eq:order}
				d_i \leq d_j, \forall 1\leq j < i \leq k.
			\end{equation}
			
			We will show the assertion by an induction on the length $k$ of $\mathcal{F}$.
			
			
			Consider the subset $J$ of $\{ 1,\dots, k\}$ defined by
			\[ J= \{ 1\leq j \leq k | f_{ij} \text{ is a nonzero multiple of the identity morphism for some }i>j\}. \]
			
			If $J$ is empty, then $\mathcal{F}$ is the desired twisted complex.

			Now assume that $J$ is non-empty. We will find another twisted complex for $\mathcal{F}$ whose length is strictly less than that of $\mathcal{F}$. For that purpose, let $j_0 = \mathrm{max} J$ and consider the twisted complex $\mathcal{F}_0$ given by
			\[ \mathcal{F}_0 = \left[ (L_{v_i}[d_i])_{i=j_{0}}^{k}, (f_{ij}) \right]\]
			and the twisted complex $\mathcal{F}_1$ given by
			\[ \mathcal{F}_1 = \left[(L_{v_i}[d_i])_{i=1}^{j_{0}-1}, (f_{ij}) \right]\]
			so that we can express $\cF$, using the notation in Remark \ref{rmk iterated cone}\eqref{item:notation}, as the twisted complex
			\[\cF=\left[\cF_1\to\cF_0 \right] .\]
			Hence, it suffices to find a twisted complex $\mathcal{E}_0$ for $\mathcal{F}_0$ whose length is strictly less than that of $\mathcal{F}_0$ and replace $\mathcal{F}_0$ by $\mathcal{E}_0$ in the above twisted complex.
			
			For that purpose, let $i_0 >j_0$ be an integer such that $f_{i_0,j_0}$ is a nonzero multiple of the identity morphism, which means that $d_{i_0} = d_{j_0} -1$. But, due to \eqref{eq:degree}, this means that there is no nonzero chain of arrows connecting $L_{v_{j_0}}[d_{j_0}]$ to $L_{v_{i_0}}[d_{i_0}]$ other than the above arrow $f_{i_0,j_0}$. Furthermore, the maximality of $j_0$ in $J$ and \eqref{eq:order} imply that there is no nonzero arrow from $L_{v_j} [d_j]$ to both $L_{v_{j_0}}[d_{j_0}]$ and $L_{v_{i_0}}[d_{i_0}]$ for any $j > j_0$. Therefore, $\mathcal{F}_0$ can be written as
			\[\cF_0=[\mathcal{E}_1\to\mathcal{E}_0]\]
			where
			\[ \mathcal{E}_0 = \left[ (L_{v_i}[d_i])_{j_0 < i \leq k, i_0 \neq i \neq j_0}  , (f_{ij}) \right],\]
			and
			\[\mathcal{E}_1=\left[L_{v_{j_0}}[d_{j_0}]\xrightarrow{f_{i_0,j_0}} L_{v_{i_0}}[d_{i_0}]\right].\]
			
			But the twisted complex $\mathcal{E}_1$ is quasi-isomorphic to zero as $f_{i_0,j_0}$ is a nonzero multiple of the identity morphism. Hence $\mathcal{F}_0$ is quasi-isomorphic to $\mathcal{E}_0$, whose length is strictly less than that of $\mathcal{F}_0$ as desired.
			
			This induction process eventually stops since the number of components of a twisted complex is finite.
	\end{proof}
	
	\begin{rmk}
		The second statement (2) of Theorem \ref{thm:simpletwistedcomplex} can be generalized to any cohomological unital pretriangulated $A_{\infty}$ (or dg) category $\cC$ generated by $G_1,\dots, G_m$ satisfying
		\begin{itemize}
			\item $\Hom^*_{\cC}(G_i,G_j)$ is non-positively graded and
			\item $\dim \Hom^0_{\cC} (G_i,G_i) = 1$ for all $i=1,\dots,m$
		\end{itemize}
		in the sense that, any object of $\cC$ is quasi-isomorphic to a twisted complex built from $G_1,\dots,G_m$, in which none of arrows are multiples of the identity morphisms. Here, by an identity morphism, we mean a strict unit for $G_i$ for some $i=1,\dots, k$ after replacing $\bigoplus_{i,j =1}^k \hom_{\cC} (G_i,G_j)$ with a minimal model.
	\end{rmk}

	We remark that the twisted complex mentioned in Theorem \ref{thm:simpletwistedcomplex}, (1) is not only an object of $\F(P_n(T))$, but also an object of $\W(P_n(T))$.
	In the statement/proof of Theorem \ref{thm:simpletwistedcomplex}, we consider the twisted complexes as an object in $\W(P_n(T))$ and the hom-spaces below are morphism spaces of $\W(P_n(T))$.
	
	For each $m \in \mathbb{N}$, let $\mathcal{S}_m \in \Tw \{S_v|v \in V(T)\}$ be a twisted-complex for $\phi^m(S)$ satisfying the property in Theorem \ref{thm:simpletwistedcomplex} (1), and let $\mathcal{L}_m \in \Tw \{L_v| v \in V(T)\}$ be a twisted-complex for $\phi^m(L)$ satisfying the property in Theorem \ref{thm:simpletwistedcomplex} (2). Then we have the following:
	
	\begin{lem}\label{lem:simpleness}
		Let $\phi$ be a Liouville automorphism of $P_n(T)$.  For any $m \in \mathbb{N}$ and $v \in V(T)$, the following hold:
		\begin{enumerate}
			\item The cochain complex $\hom^*(\mathcal{S}_m, L_v)$ has the zero differential map.
			\item The cochain complex $\hom^*(S_v, \mathcal{L}_m)$ has the zero differential map.
		\end{enumerate}
	\end{lem}
	\begin{proof}
	Since the proof for the second statement is similar to that of the first statement, we will prove the first statement only.
			
			The twisted complex $\mathcal{S}_m$ may be written as
			\[ \mathcal{S}_m = \left[(S_{v_i} [d_i])_{i=1}^{a}, (f_{i j})\right] \]
			where $(f_{ij} \in \hom^1( S_{v_j} [d_j], S_{v_i}[d_i])$, $i > j)$ are the arrows in the twisted complex $\mathcal{S}_m$.
			
			Since the components of $\mathcal{S}_m$ that contribute to $\hom^*(\mathcal{S}_m, L_v)$ are of the form $S_v[d]$, we collect all the indices $1\leq i_1 < \dots < i_{ \len_{S_v} \mathcal{S}_m} \leq a$ such that $v_{i_k} = v$. Then the morphism space $\hom^*(\mathcal{S}_m,L_v)$ is given by
			\[ \hom^*(\mathcal{S}_m,L_v) = \bigoplus_{1\leq k \leq \len_{S_v}\mathcal{S}_m} \hom^*(S_v[d_{i_k}],L_v) \simeq \bigoplus_{1\leq k \leq \len_{S_v}\mathcal{S}_m} k[-d_{i_k}-n]\]
			as a graded vector space.
			
			The differential from a summand $\hom^*(S_v[d_{i_k}],L_v) \simeq k[-d_{i_k}-n]$ to itself vanishes since it is one-dimensional. Now suppose that the differential from a summand  $\hom^*(S_v[d_{i_l}],L_v) \simeq k[-d_{i_l}-n]$ to $\hom^*(S_v[d_{i_k}],L_v)\simeq k[-d_{i_k} -n]$ does not vanish for some $1\leq k < l \leq \len_{S_v}\mathcal{S}_m$.
			First of all, since the differential increases degree by 1, it means that
			\begin{equation}\label{degree}
				d_{i_k} = d_{i_l}+1.
			\end{equation}
			On the other hand, it also implies that there is a sequence $i_k = j_0 < j_1 <\dots <j_p= i_l $ for some $p \in \Z_{\geq 1}$ such that the arrow $f_{j_{q+1} ,j_{q}} \in \hom^1(S_{v_{j_q}}[d_{j_q}], S_{v_{j_{q+1}}}[d_{j_{q+1}}])$ is nonzero for all $0\leq q <p$. But this means that
			\begin{equation*}
				d_{j_{q+1}} = d_{j_q} + \deg f_{j_{q+1},j_{q}} - 1 \geq d_{j_q}.
			\end{equation*}
			for all $0\leq q <p$, where $\deg f_{j_{q+1},j_{q}}$ is the degree of $f_{j_{q+1},j_{q}}$ as a morphism from $S_{v_{j_q}}$ to $S_{v_{j_{q+1}}}$. Indeed, $\deg f_{j_{q+1},j_{q}} \geq 1$ since $f_{j_{q+1},j_{q}}$ is not a multiple of the identity morphism by our assumption and any homogeneous element of $\hom^* (S_v,S_w)$ for $v,w \in V(T)$ has a positive degree unless it is a multiple of the identity morphism. 
			
			As a result, we have $d_{i_l} \geq d_{i_k}$, which contradicts to \eqref{degree}.
	\end{proof}

	Now we are ready to prove the following:
	\begin{thm}\label{thm:entropy}
		Let $\phi$ be a Liouville automorphism of $P_n(T)$. For any $t\in \R$, the following hold:
		\begin{gather*}
			h_t \left( \Phi_{\mathcal{F} (P_n(T))}\right) = \lim_{m \to \infty} \frac{1}{m} \log  \sum_{k \in \Z} \dim \Hom^{n+k} \left(\phi^m(S), L\right) e^{kt}, \\
			h_t \left( \Phi_{\mathcal{W} (P_n(T))}\right) = \lim_{m \to \infty} \frac{1}{m} \log  \sum_{k\in \Z} \dim \Hom^{n-k} \left(S, \phi^m(L)\right) e^{kt}.
		\end{gather*}
	\end{thm}
	\begin{proof}
		Since the proof for the second equality is similar to that of the first equality, we will prove the first one only.
		
		Recall that the categorical entropy $h_t(\Phi_{\F(P_n(T))})$ measures the exponential growth rate of
		\begin{equation}
			\label{infimum}
			\text{inf }\left\{  \sum_{v \in V(T)}  \len_{t,S_v} \widetilde{\mathcal{S}}_m \; \Bigg| \; \widetilde{\mathcal{S}}_m \simeq \phi^m(S) \oplus C \text{ for some } \widetilde{\mathcal{S}}_m \in \Tw \{S_v|v\in V(T)\}, C \in \F(P_n(T)) \right\}.
		\end{equation}

		Let $\mathcal{S}_m$ be a twisted complex for $\phi^m(S)$ given right above of Lemma \ref{lem:simpleness}.
		Then, for any twisted complex $\widetilde{\mathcal{S}}_m$ quasi-isomorphic to $\phi^m(S) \oplus C$ for some $C \in \F(P_n(T))$, we have
		\begin{equation*}
			\begin{split}
				\len_{t,S_v} \widetilde{\mathcal{S}}_m &=  \sum_{k \in \Z} \dim \hom^{n+k} \left( \widetilde{\mathcal{S}}_m, L_v\right) e^{kt} \hspace{0.5em} (\because \text{Lemma } \ref{lem:length1} ) \\
				&\geq \sum_{k \in \Z} \dim \Hom^{n+k} \left( \phi^m(S) \oplus C , L_v\right) e^{kt} \\
				&\geq \sum_{k \in \Z} \dim \Hom^{n+k}\left(\phi^m(S), L_v\right) e^{kt}\\
				&= \sum_{k \in \Z} \dim \hom^{n+k} \left(\mathcal{S}_m, L_v\right) e^{kt} \hspace{0.5em} (\because \text{Lemma } \ref{lem:simpleness})\\
				&= \len_{t,S_v} \mathcal{S}_m \hspace{0.5em} (\because \text{Lemma } \ref{lem:length1}).
			\end{split}
		\end{equation*}
		This means that $\mathcal{S}_m$ is a twisted complex quasi-isomorphic to $\phi^m(S)$ giving the infimum in \eqref{infimum}.
		
		Therefore, the categorical entropy $h_t( \Phi_{\mathcal{F} (P_n(T)}))$ is computed by the exponential growth rate of 
		\[\sum_v \len_{t,S_v} \mathcal{S}_m = \sum_v \sum_{k\in \Z} \dim \hom^{n+k}( \mathcal{S}_m, L_v)e^{kt} = \sum_v \sum_{k \in \Z} \dim \Hom^{n+k} \left(\phi^m(S), L_v\right) e^{kt} = \sum_{k \in \Z} \Hom^{n+k}\left(\phi^m(S),L\right)e^{kt} \]
		as asserted.
	\end{proof}

	\subsection{Equality between categorical entropies}
	\label{subsection equality between categorical entropies}
	Let $n\geq 3$. We prove that, for any compactly supported exact symplectic automorphism $\phi$ on $P_n(T)$, their induced functors $\Phi_{\W(P_n(T))}$ and $\Phi_{\F(P_n(T))}$ have the same categorical entropy under certain conditions.
	
	\begin{thm}\label{thm:equality}
		Let $\phi$ be a Liouville automorphism of $P_n(T)$. Then we have
		\begin{enumerate}
			\item For any $t\in \R$, the following equality hold:
			\begin{align*}
				h_{t} \left(\Phi_{\F(P_n(T))}\right) = h_{-t}\left(\Phi^{-1}_{\W(P_n(T))}\right), h_{t} \left(\Phi_{\W(P_n(T))}\right) = h_{-t}\left(\Phi^{-1}_{\F(P_n(T))}\right)
			\end{align*}
			\item If $\phi$ is further compactly supported, then, for any $t\in \R$ such that $h_t \left(\Phi_{\F(P_n(T))}\right) \geq 0$ and $h_t \left(\Phi_{\W(P_n(T))}\right) \geq 0$, the following equality holds:
			\begin{equation*}
				h_t \left(\Phi_{\F(P_n(T))} \right) 
			=h_t \left( \Phi_{\W(P_n(T))}\right).
			\end{equation*}
			In particular, since both $h_0 \left(\Phi_{\F(P_n(T))}\right) \geq 0$ and $h_0 \left(\Phi_{\W(P_n(T))}\right) \geq 0$  always hold, the following equality holds:
			\begin{equation*}
				h_0 \left(\Phi_{\F(P_n(T))} \right) 
			=h_0 \left( \Phi_{\W(P_n(T))}\right).
			\end{equation*}
		\end{enumerate}
	\end{thm}
	\begin{proof}
		Theorem \ref{thm:entropy} implies that, for any $t\in \R$,
		\begin{align*}
			h_{t}\left(\Phi_{\F(P_n(T)}\right) &= \lim_{m \to \infty} \frac{1}{m} \log \sum_{k\in \Z} \dim \Hom^{n+k} \left(\phi^m(S), L\right) e^{kt}\\
			&=  \lim_{m \to \infty} \frac{1}{m} \log \sum_{k\in \Z} \dim \Hom^{n+k} \left( S, \phi^{-m} ( L)\right) e^{kt}\\
			&= \lim_{m \to \infty} \frac{1}{m} \log \sum_{k\in \Z} \dim \Hom^{n-k} \left( S, \phi^{-m} ( L)\right) e^{k\cdot (-t)}\\
			&= h_{-t}\left(\Phi^{-1}_{\W(P_n(T))}\right). 
		\end{align*}
		The other equality can be proved similarly. This completes the proof of the first statement.
		
		Then the second statement follows from Theorem \ref{thm:entropy-w-pw} and the first statement.
	\end{proof}
	
	
	More generally, Theorems \ref{thm:entropy} and \ref{thm:equality} hold for any Weinstein manifold $W$ and any Liouville automorphism $\phi$ of $W$ if $W$ satisfies the conditions in the following theorem.
	\begin{thm}
		\label{thm general}
		Assume that a Weinstein manifold $W$ has sets of Lagrangians $\{S_i\}_{1 \leq i \leq k}$ and $\{L_i\}_{1 \leq i \leq k}$ such that 
		\begin{itemize}
			\item $\{S_i\}_{1 \leq i \leq k}$ (resp.\ $\{L_i\}_{1 \leq i \leq k}$) generates the compact (resp.\ wrapped) Fukaya category of $W$,
			\item the morphism space in the wrapped Fukaya category $\hom(S_i,L_j)$ is non-zero only if $i=j$ and is one-dimensional when $i=j$,
			\item $\bigoplus_{i,j} \hom^* (S_i,S_j)$ is non-negatively graded, 
			\item  $\bigoplus_{i,j} \hom^* (L_i,L_j)$ is non-positively graded,
			\item $\dim \Hom^0 (S_i,S_i) = 1 = \dim \Hom^0 (L_i,L_i)$ for all $1\leq i \leq k$, and
			\item $\dim \Hom^0(S_i, S_j) =0$ for $i\neq j$.
		\end{itemize}
		Let $S := \bigoplus_{i=1}^k S_i, L := \bigoplus_{i=1}^k L_i$.
		If $\phi$ is a Liouville automorphism of $W$, then the following holds:
		\begin{enumerate}
			\item[(1)] For any $t \in \mathbb{R}$, 
			\begin{gather*}
				h_t \left( \Phi_{\mathcal{F} (W)}\right) = \lim_{m \to \infty} \frac{1}{m} \log  \sum_{k \in \Z} \dim \Hom^{n+k} \left(\phi^m(S), L\right) e^{kt}, \\
				h_t \left( \Phi_{\mathcal{W} (W)}\right) = \lim_{m \to \infty} \frac{1}{m} \log  \sum_{k\in \Z} \dim \Hom^{n-k} \left(S, \phi^m(L)\right) e^{kt}.
			\end{gather*}
			\item[(2)] For any $t \in \mathbb{R}$,
			\begin{align*}
				h_{t} \left(\Phi_{\F(W)}\right) = h_{-t}\left(\Phi^{-1}_{\W(W)}\right), h_{t} \left(\Phi_{\W(W)}\right) = h_{-t}\left(\Phi^{-1}_{\F(W)}\right)
			\end{align*}
			\item[(3)] If $\phi$ is further compactly supported, then, for any $t \in \mathbb{R}$ such that $h_t \left(\Phi_{\F(W)}\right) \geq 0$ and $h_t \left(\Phi_{\W(W)}\right) \geq 0$,
			\begin{equation*}
				h_{t} \left(\Phi_{\F(W)}\right) = h_{t}\left(\Phi_{\W(W)}\right).
			\end{equation*}
			In particular, if $t=0$, the condition always holds. Thus, 
			\begin{equation*}
				h_0 \left(\Phi_{\F(W)}\right) = h_0\left(\Phi_{\W(W)}\right).
			\end{equation*}
		\end{enumerate}
	\end{thm}
	\begin{proof}
		In the proofs of Theorems \ref{thm:entropy} and \ref{thm:equality}, we only used the properties of $\F(P_n(T))$ and $\W(P_n(T))$ which are assumed in Theorem \ref{thm general}. 
		Thus, the proofs also work for a Weinstein manifold $W$ satisfying the assumptions.
	\end{proof}

	\begin{rmk}
		\mbox{}
		\begin{enumerate}
			\item We would like to point out that in the proof of Theorem \ref{thm general} (1) and (2), we did not use any geometric property of $\phi$ and we only used the duality between $\cW(W)$ and $\cF(W)$. 
			Thus, if a pair of categories $(\mathcal{C} \supset \mathcal{D})$ satisfies the duality conditions in Theorem \ref{thm general}, i.e., $\mathcal{C}$ admits a generator $L$ and $\cD$ admits a generator $S$ satisfying the conditions of Theorem \ref{thm general}, the proof works for any functors of $\mathcal{C}$ or $\mathcal{D}$.
			On the other hand, in the proof of Theorem \ref{thm general} (3), a Calabi-Yau property of Fukaya category is required. 
			\item We would like to introduce a result of \cite{Fan-Filip23} that can reduce the difficulty in applying Theorem \ref{thm general} (3).
			In order to apply Theorem \ref{thm general} (3), one should check that $h_t\left(\Phi_{\F(W)}\right), h_t\left(\Phi_{\W(W)})\right) \geq 0$.
			In \cite{Fan-Filip23}, Fan and Filip defined the {\em upper/lower shifting numbers $\tau^\pm(\Phi)$} for endofunctor $\Phi$ as analogues to Poincar\'e translation numbers. 
			Moreover, they proved that 
			\begin{gather*}
				 t \cdot \tau^+(\Phi) \leq h_t(\Phi) \leq h_0(\Phi) + t \cdot \tau^+(\Phi), \text{  for  } t \geq 0, \\
				 t \cdot \tau^-(\Phi) \leq h_t(\Phi) \leq h_0(\Phi) + t \cdot \tau^-(\Phi), \text{  for  } t \leq 0.
			\end{gather*}
			Thus, in order to check the conditions of Theorem \ref{thm general} (3), it would be enough to check the non-negativeness (resp.\ non-positiveness) of $\tau^+(\Phi)$ (resp.\ $\tau^-(\Phi)$) for $t \geq 0$ (resp.\ $t \leq 0$).
		\end{enumerate}
	\end{rmk}
	
	\begin{example} 
	Another simple example of Weinstein manifolds whose compact and wrapped Fukaya categories satisfy the duality condition is a cotangent bundle $T^*Q$ of a simply-connected closed smooth manifold $Q$. 
	Choose a background class as $b = \pi^*w_2(Q)$, the pullback of the second Stiefel-Whitney class of $Q$ via the projection $\pi : T^*Q \to Q$.
	Suppose $\phi$ be any Liouville automorphism of $T^*Q$ such that $\phi^* b =b$.
	Then one can show that there is an integer $d_\phi$ such that 
	\[ h_t(\Phi_{\F}) = d_\phi t= h_t(\Phi_{\W} ).\]
	It is a simple corollary of \cite{Nadler09,Fukaya-Seidel-Smith,Abouzaid12}.
	The proof is left to the reader.
	\end{example}

	\section{Entropies of symplectic automorphisms of Penner type}
	\label{section entropies of symplectic automorphisms of Penner type}
	Let $\phi$ be a compactly supported exact symplectic automorphism on $P_n(T)$, $n \geq 3$.
	Then, since the induced functors on the compact and wrapped Fukaya categories have the same categorical entropies for $t=0$ by Theorem \ref{thm:equality}, we simply say that $\phi$, not the induced functors, has a categorical entropy. 
	More precisely,
	\begin{dfn}
		\label{def categorical entropy of phi}
		If $\phi$ is any compactly supported exact symplectic automorphism on $P_n(T)$, then the {\em categorical entropy of $\phi$} is 
		\[h_{cat}(\phi) := h_0\left(\Phi_{\F(P_n(T))}\right) = h_0\left(\Phi_{ \W(P_n(T))}\right).\]
	\end{dfn}
	
	The result of Section \ref{section entropies of products of Dehn twists} says that
		\[ h_{cat}(\phi) = \lim_{m \to\infty} \frac{1}{m} \log \sum_{k \in \mathbb{Z}} \dim \Hom^k\left(\phi^{m}(S), L\right).\]
		Even if we have the above result, $h_{cat}(\phi)$ is not easy to compute because measuring the exponential growth rate of $\dim \Hom^*\left(\phi^{m}(S), L\right)$ is not an easy task in general. In this section, we prove that, if $\phi$ is of a symplectic automorphism of specific type on $P_n(T)$, its categorical entropy is related to the spectral radius of the matrix $M_{\phi}$ whose entries are given by the dimensions of $\Hom$ spaces. The specific type will be defined in Section \ref{subsection symplectic automorphisms of Penner type}.
		These results will more precisely stated in Theorem \ref{thm:entropyofpennertype}.
		Combined with the recent works \cite{Bae-Lee22, Bae-Lee23}, Section \ref{section entropies of symplectic automorphisms of Penner type} provides infinitely many Torelli group elements having positive topological (and categorical) entropy. 
		See Remark \ref{rmk torelli}. 
	
	\subsection{Symplectic automorphisms of Penner type}
	\label{subsection symplectic automorphisms of Penner type}
	Section \ref{section entropies of products of Dehn twists} cares any compactly supported exact symplectic automorphism on $P_n(T)$, but as mentioned in the beginning of Section \ref{section entropies of symplectic automorphisms of Penner type}, here we care symplectic automorphisms of a specific type, called {\em Penner type}.
	
	In order to define the notion of Penner type, we consider the following: 
	Let $T$ be a tree, and let $V(T)$ be the set of vertices. 
	Then, by choosing a vertex $v_0 \in V(T)$, one can define the following:
	\begin{equation}\label{eqn V(T)} 
		\begin{split}
			&V_+(T) := \{ v \in V(T) \hspace{0.2em} | \hspace{0.2em} v \text{  is connected to  } v_0 \text{  by even number of edges.}\} \\
			&V_-(T) := \{ v \in V(T) \hspace{0.2em} | \hspace{0.2em} v \text{  is connected to  } v_0 \text{  by odd number of edges.}\}
		\end{split}
	\end{equation}
	The above $V_+(T)$ and $V_-(T)$ are disjoint decomposition of $V(T)$. 
	
	\begin{dfn}
		\label{def Penner type}
		Let $\tau_v$ denote the Dehn twist along $S_v$.
		{\em A symplectic automorphism $\phi: P_n(T) \to P_n(T)$ is of Penner type} if either $\phi$ or $\phi^{-1}$ is a product of 
		\begin{itemize}
			\item positive powers of $\tau_v$ if $v \in V_+(T)$, and
			\item negative powers of $\tau_w$ if $w \in V_-(T)$.
		\end{itemize}
	\end{dfn}
	
	\begin{rmk}
		\mbox{}
		\begin{enumerate}
			\label{rmk welldefinedness of Penner type}
			\item This type of symplectic automorphisms is studied by the last named author in \cite{Lee}. 
			In \cite{Lee}, it is proved that symplectic automorphisms of Penner type satisfy a geometric stability. 
			Because of the geometric stability, we could expect that one could easily compute the categorical entropy of Penner type by a simple linear algebra.
			\item We would like to point out that two sets $V_+(T)$ and $V_-(T)$ defined in Equation \eqref{eqn V(T)} are dependent on the choice of $v_0$, but Definition \ref{def Penner type} is independent from the choice. 
			In the rest of this section, we assume there is an arbitrary chosen $v_0$.
			\item We note that the original Penner construction \cite{Penner} in surface theory has one more condition than Definition \ref{def Penner type}. 
			The extra condition is the following Condition (P).
			\begin{itemize}
				\item[(P)] every $v \in V(T)$ appears in the product. 
			\end{itemize} 
			We omit the condition (P) since the results in the rest of the paper hold without the condition (P).
				But, if $\phi$ does not satisfy (P), we expect that $\phi$ does not induce a pseudo-Anosov auto-equivalence in the sense of \cite{Fan-Filip-Haiden-Katzarkov-Liu}.
			For more details on surface theory, we refer the eager reader to \cite{Farb-Margalit}.
		\end{enumerate}
	\end{rmk}
	
	For the future use, we set the following notations. First, let us denote an ordered sequence of vertices $v_i \in V(T)$ by $( v_s, v_{s-1}, \dots, v_1 )$.
	
	\begin{dfn}
		\label{def sign etc}
		\mbox{}
		\begin{enumerate}
			\item For each $v \in V(T)$, let {\em $\sigma_v$} be defined as follows:
			\[\sigma_v  = \begin{cases}
				1 &\text{  if  } v \in V_+(T),\\
				-1 &\text{  if } v \in V_-(T).
			\end{cases}\]
			\item Let $J = (v_s, v_{s-1}, \dots, v_1)$ be an ordered sequence of vertices $v_i \in V(T)$.
			Then, {\em $\phi_J$} (resp.\ {\em $\phi_{-J}$}) is the symplectic automorphism of Penner type defined as follows:
			\[\phi_J= \tau_{v_s}^{\sigma_{v_s}} \circ \tau_{v_{s-1}}^{\sigma_{v_{s-1}}} \circ \dots \circ \tau_{v_1}^{\sigma_{v_1}}, \hspace{0.5em} \phi_{-J}= \tau_{v_s}^{-\sigma_{v_s}} \circ \tau_{v_{s-1}}^{-\sigma_{v_{s-1}}} \circ \dots \circ \tau_{v_1}^{-\sigma_{v_1}}.\]
		\end{enumerate}
	\end{dfn}
	We note that by Definition \ref{def Penner type}, for any $\phi$ of Penner type, there is an ordered sequence $J$ such that $\phi = \phi_J$ or $\phi = \phi_{-J}$. However, by changing a choice of $v_{0}$, one can always assume that $\phi = \phi_J$; see Remark \ref{rmk welldefinedness of Penner type}, (2).
	In the rest of the paper, we assume that $\phi = \phi_J$ for some $J$.

	\subsection{Construction of a twisted complex}
	\label{subsection construction of a twisted complex}
	In Section \ref{section entropies of products of Dehn twists}, we constructed a twisted complex in Theorem \ref{thm:simpletwistedcomplex} whose nonzero arrows are not a multiple of the identity morphism.  This property of the twisted complex helped to prove Theorem \ref{thm:equality}.
	Now, we construct a twisted complex with the same property, by using Seidel's long exact sequence \cite{Seidel03} in the wrapped Fukaya category, which will help us to prove the main theorem of this section.
	
	\begin{lem}\label{lem:dehntwist}
		\mbox{}
		\begin{enumerate}
			\item In the wrapped Fukaya category $\W(P_n(T))$, $\tau_v(S_w)$ is isomorphic to the following twisted complex:
			\begin{equation} \label{twistedcomplex1}
				\tau_v(S_w) \simeq \begin{cases}  S_w [1-n] &  v = w,\\ 
					\left[(S_v[1-s_{vw}], S_{w})  , f \right] & v \sim w, \\
					S_{w} & \mbox{otherwise.}
				\end{cases}
			\end{equation}
			\item In the wrapped Fukaya category $\W(P_n(T))$, $\tau_v^{-1}(S_w)$ is isomorphic to the following twisted complex:
			\begin{equation} \label{twistedcomplex2}
				\tau_v^{-1}(S_w) \simeq \begin{cases}  S_w [n-1] &  v = w,\\ 
					\left[(S_w, S_v[n-s_{vw}-1])  , g \right] & v \sim w, \\
					S_{w} & \mbox{otherwise.}
				\end{cases}
			\end{equation}
		\end{enumerate}
	\begin{proof}
				All the cases are from the Seidel's long exact sequence and the grading convention in Lemma \ref{lemma grading}.
		\end{proof}
	\end{lem}
	Note that $f$ and $g$ in \eqref{twistedcomplex1} and \eqref{twistedcomplex2} can be described explicitly, but we do not since they will not be used. 
	For the future convenience, let 
	\begin{equation}
		\label{eqn good twisted complex}
		\mathcal{S}_{m,v} = \left[ (S_w [d_i])_{\substack{w \in V(T) \hspace{12mm} \\ 1\leq i \leq \len_{S_w} (\mathcal{S}_{m,v})}}, \left(f_{ij} \right)\right]
	\end{equation}
	denote the twisted complex for $\phi^{m}(S_{v})$. 
	
	If we grade $S_{v}$'s so that $s_{uw} = 1$, or equivalently, $s_{wu}=n-1$ for any $u \in V_{+}(T)$, $w \in V_{-}(T)$, then we get the following lemma.
	\begin{lem}\label{lem:shift}
			With the notation given above, the shift $d_{i}$ of $S_w[d_{i}]$ in Equation \eqref{eqn good twisted complex} satisfies that
			\[d_{i} = l (1-n), \text{  for some  } l \in \mathbb{Z}.\]
	\end{lem}
	
	\subsection{Entropies of symplectic automorphisms of Penner type}
	\label{subsection entropies of symplectic automorphisms of Penner type}
	In this subsection, we prove the main theorem of Section \ref{section entropies of symplectic automorphisms of Penner type} (Theorem \ref{thm:entropyofpennertype}).
	It says that if a symplectic automorphism $\phi$ is of Penner type, then the categorical entropy $h_{cat}(\phi)$ can be computed by a simple linear algebra. 
	
	In order to state Theorem \ref{thm:entropyofpennertype}, we define the notion of spectral radius, first.
	\begin{dfn}
		\label{def spectral radius}
		Let $F : V \to V$ be a linear operator on a finite dimensional $\R$-vector space $V$. 
		The {\em spectral radius $\text{Rad}(F)$} of $F$ is defined by the maximum of absolute values of the complex eigenvalues of $F$. 
	\end{dfn}
	
	We state and prove the main theorem of Section \ref{section entropies of symplectic automorphisms of Penner type}.
	
	\begin{thm}\label{thm:entropyofpennertype}
		Let $\phi$ be a symplectic automorphism of Penner type on $P_n(T)$ with $n \geq 3$ and
		$M_{\phi} : \mathbb{R}^{|V(T)|} \to \mathbb{R}^{|V(T)|}$ be a matrix whose $uw$-entry is given by $\displaystyle \sum_{k} \dim \Hom^k\left(\phi(S_w),L_u \right)$.
			Then,
			\begin{equation*}
				h_{cat}(\phi) = \log \text{Rad} (M_{\phi}).
		\end{equation*}
	\end{thm}
	
	\begin{proof}
		We note that for convenience, we use the set $V(T)$ of vertices as an index set throughout the proof. 
		
		For any $v \in V(T)$, $\tau_v^{\sigma_v}$ is of Penner type. 
		Thus, $M_{\tau_v^{\sigma_v}}$ is defined and it satisfies that 
		\[\left(M_{\tau_v^{\sigma_v}}\right)_{uw} = \begin{cases}
			1 & u=w, \\
			1 & u=v, w \sim v, \\
			0 & \mbox{otherwise,}
		\end{cases}\]
		because of Lemma \ref{lem:dehntwist}. 
		
		We note that for any twisted complex $\mathcal{S} \in \Tw\{S_v | v \in V(T)\}$, Lemma \ref{lem:dehntwist} gives another twisted complex $\mathcal{S}' \simeq \tau_v^{\sigma_v}\left(\mathcal{S}\right)\in \Tw\{S_v | v \in V(T)\}$.
		Moreover, the following equation holds:
		\begin{gather}
			\label{eqn linear algebra}
			M_{\tau_v^{\sigma_v}} \cdot \left(\len_{S_w}\mathcal{S}\right)_{w \in V(T)} = \left(\len_{s_w}\mathcal{S}'\right)_{w \in V(T)}, \text{  where  } \left(\len_{S_w}\mathcal{S}\right)_{w \in V(T)}, \left(\len_{s_w}\mathcal{S}'\right)_{w \in V(T)} \in \mathbb{R}^{|V(T)|}.
		\end{gather}
		
		Now, we focus on the twisted complex $\mathcal{S}_{m,v}$ defined in Equation \eqref{eqn good twisted complex}. 
		Lemma \ref{lem:shift} implies that every arrow of $\mathcal{S}_{m,v}$ is not a scalar multiple of the identity morphism for any $m \in \mathbb{N}, v \in V(T)$. 
		Thus, $\hom^*\left(\mathcal{S}_{m,v},L_w\right)$ has the zero differential, and 
		\begin{equation}\label{eq:len-dim}
			\len_{S_w} \mathcal{S}_{m,v} = \dim \hom^*\left(\mathcal{S}_{m,v}, L_w\right) = \dim \Hom^*\left(\phi^m(S_v),L_w\right) \text{  for all  } m \in \mathbb{N}, v, w \in V(T).
		\end{equation}
		Thus, the $(w,v)$-entry of $M_{\phi^m}$ is given as 
		\begin{gather}
			\label{eqn each entry}
			\left(M_{\phi^m}\right)_{w,v} = \sum_{k \in \mathbb{Z}} \dim \hom^k\left(\mathcal{S}_{m,v}, L_w\right) = \len_{S_w} \mathcal{S}_{m,v}.
		\end{gather}
		We note that since $\phi^m$ is also of Penner type, $M_{\phi^m}$ is defined. 
		
		Let $\phi = \phi_J$ with some $J = (v_s, \dots, v_1)$, i.e., 
		\[\phi_J= \tau_{v_s}^{\sigma_{v_s}} \circ \tau_{v_{s-1}}^{\sigma_{v_{s-1}}} \circ \dots \circ \tau_{v_1}^{\sigma_{v_1}}.\]
		If $M_i$ denotes $M_{\tau_{v_i}^{\sigma_{v_i}}}$, Equations \eqref{eqn linear algebra} and \eqref{eqn each entry} conclude that 
		\begin{equation}\label{eq:matrix-properties}
			M_\phi = M_s \cdots M_2 \cdot M_1,\quad M_{\phi^m} = M_\phi^m.
		\end{equation}
		
		For any $v \in V(T)$, let $e_v \in \mathbb{R}^{|V(T)|}$ be the vector whose $v^\text{th}$-entry is $1$ and the other entries are zero. 
		And let $\parallel \cdot \parallel_1$ mean the $L^1$-norm on $\mathbb{R}^{|V(T)|}$.
		Theorem \ref{thm:entropy}, Equations \eqref{eq:len-dim}, \eqref{eq:matrix-properties}, and \eqref{eqn each entry} conclude that 
		\begin{equation}
			\label{eqn entropy computation}
			\begin{split}
				h_{cat}(\phi) &= \lim_{m \to \infty} \frac{1}{m} \log \sum_{v, w \in V(T)} \len_{S_w}\mathcal{S}_{m,v} \\
				&= \lim_{m \to \infty} \frac{1}{m} \log \max \left\{\parallel M_\phi^m \cdot e_v \parallel_1 \bigg| v \in V(T) \right\}
			\end{split}
		\end{equation}
		because every entry of $M_\phi$ is non-negative and $\{e_v | v \in V(T)\}$ is a basis of $\mathbb{R}^{|V(T)|}$. 
		Thus, by Gelfand's formula, the last term in Equation \eqref{eqn entropy computation} is the same as $\log \text{Rad}(M_\phi)$. 
	\end{proof}
	
	\begin{cor}
		\label{cor cat vs top}
		Let $\phi$ be a symplectic automorphism of Penner type, and let $n \geq 3$ be an odd integer.
		Then, 
		\[h_{cat}(\phi) \leq h_{top}(\phi),\]
		where $h_{top}(\phi)$ means the topological entropy of $\phi$.
	\end{cor}
	\begin{proof}
		When $n$ is odd, it is easy to see that the matrix representation of $\phi_* : H_n(P_n(T)) \to H_n(P_n(T))$ with respect to the basis $\{[S_{v}]\}_{v \in V(T)}$ and $M_{\phi}$ coincide. Moreover, it is well-known that 
			\[\log \text{Rad} \left( \phi_* : H_n(P_n(T)) \to H_n(P_n(T)) \right) \leq h_{top}(\phi).\]
			See \cite{Yomdin, Gromov1, Gromov2} for details. 
			Then, Theorem \ref{thm:entropyofpennertype} completes the proof.
	\end{proof}
	
	
	\subsection{A counterexample to Gromov--Yomdin type equality}
	\label{subsection Gormov Yomdin type equality}
	We give a counter-example of the Gromov--Yomdin type equality using Theorem \ref{thm:entropyofpennertype}.
	The equality claims that the categorical entropy of a functor $\Phi$ is the same as the logarithm of the spectral radius of the linear map which $\Phi$ induces on the Grothendieck group. 
	For more details on Gromov--Yomdin type equality, we refer the reader to \cite[Section 1.2]{Kikuta-Ouchi} 
	
	\subsubsection{First example}
	Let $T$ be the Dynkin diagram of $A_3$ type.
	We label the vertices of $T = A_3$ as $V(A_3) = \{1,2,3\}$ so that the vertex $1$ (resp.\ $3$) is connected to the vertex $2$.
	For a fixed $n \in \mathbb{N}_{\geq 3}$, let $\phi$ be the symplectic automorphism given as 
	\[\phi = \tau_1 \circ \tau_2^{-1} \circ \tau_3 : P_n(A_3) \stackrel{\simeq}{\to} P_n(A_3).\]
	
	Since $S_1, S_2$ and $S_3$ generate the compact Fukaya category of $P_n(A_3)$, $\{[S_1], [S_2], [S_3]\}$ generates the Grothendieck group of $\F(P_n(A_3))$. 
	By using Seidel's long exact sequences for Dehn twist $\tau_i$, one can have a linear map on the Grothendieck group, which is induced from $\tau_i$. 
	The matrix representations for $\tau_1, \tau_2^{-1}, \tau_3$ with respect to the basis $\{[S_1], [S_2], [S_3]\}$ are 
	\[B_1^+ := \begin{pmatrix}
		(-1)^{1-n} & 1 & 0 \\
		0 & 1 & 0 \\
		0 & 0 & 1
	\end{pmatrix}, 
	B_2^- := \begin{pmatrix}
		1 & 0 & 0 \\
		1 & (-1)^{n-1} & 1 \\
		0 & 0 & 1
	\end{pmatrix}, 
	B_3^+ := \begin{pmatrix}
		1 & 0 & 0 \\
		0 & 1 & 0 \\
		0 & 1 & (-1)^{1-n}
	\end{pmatrix}.
	\]
	We note that the matrix representations can vary by shifting $S_1, S_2, S_3$.
	To be more clear, we should have specified the grading information, but we omit that for convenience. 
	
	From the above computations, one can obtain a matrix representation of the linear map induced from $\phi$ on the Grothendieck group. 
	The resulting matrix is 
	\begin{gather*}
		B_\phi = B_1^+ \circ B_2^- \circ B_3^+ 
		= \begin{pmatrix}
			(-1)^{1-n}+1 &1+(-1)^{n-1} & (-1)^{n-1} \\
			1 & 1 + (-1)^{n-1} & (-1)^{n-1} \\
			0 & 1 & (-1)^{1-n}
		\end{pmatrix}.
	\end{gather*}

	If $n$ is even, then one can check by simple computation that the spectral radius of $B_\phi$ is $1$.
	Meanwhile, by Theorem \ref{thm:entropyofpennertype}, the categorical entropy of $\phi$ is equal to
	\[h_{cat}(\phi) = \log \text{Rad}M_\phi = \log \text{Rad}
	\begin{pmatrix}
		2 & 2 & 1 \\
		1 & 2 & 1 \\	
		0 & 1 & 1
	\end{pmatrix} = \log\left(2+\sqrt 3\right).\]
	Thus, the above example disproves Gromov--Yomdin type equality. 
	
	\begin{rmk}
		\label{rmk even dimensional case}
		\mbox{}
		\begin{enumerate}
			\item The previous computation also shows that there is no complex structure on $P_n(A_3)$ making $\phi$ holomorphic when $n$ is even. See \cite{Yomdin}.
			\item We note that \cite[Section 3.1]{Fan-Filip-Haiden-Katzarkov-Liu} computed the categorical entropy of symplectic automorphism of Penner type $\phi : P_n(A_2) \to P_n(A_2)$, where $A_2$ when $n \geq 3$ is an odd integer. To do that, they showed that the spectral radius equals to the categorical entropy, which is a special case of Theorem \ref{thm:entropyofpennertype} 
		\end{enumerate}
	\end{rmk}
	
	\subsubsection{Second example}
	\label{subsubsection torelli group}
	We provide another example of symplectic automorphism of Penner type that is an element of the Torelli subgroup, but induces an auto-equivalence on the compact/wrapped Fukaya category whose categorical entropy is positive. 
	
	We note that the Torelli subgroup of the mapping class group of a topological space is defined to be the subgroup consisting of mapping classes on the topological space which act trivially on its homology groups. 
	In our case, the Grothendieck group of $\cF(P_n(T))$ and $H_n(P_n(T);\mathbb{Z})$ are isomorphic via the map sending the class of Lagrangian $S_v$, $v\in V(T)$ to its homology class in the homology. 
	From this, it can be deduced that, for any symplectic automorphism on $P_n(T)$ which is an element of the Torelli subgroup, the spectral radius of its induced linear map on the Grothendieck group is 1. 

	To come up with such an example, let us recall that for any Lagrangian sphere $S$ of a symplectic manifold of dimension $4m$, $m\in \mathbb{N}$, the square of the Dehn twist along $S$ is smoothly isotopic to the identity \cite{Seidel97,Seidel99}. Now let $T$ be a tree and let $n\geq 4$ be an even integer. Then, for any $v\in V(T)$, the square $\tau_v^2$ induces the identity map on the homology since it is smoothly isotopic to the identity. Therefore, any symplectic automorphism $\phi$ of Penner type on $P_n(T)$ given by a product of $\tau_v^2$ for $v\in V_+(T)$ and $\tau_v^{-2}$ for $v\in V_-(T)$ is an element of the Torelli subgroup.
	Moreover, if the associated matrix $M_{\phi}$ (Theorem \ref{thm:entropyofpennertype}) has a positive spectral radius, then $\phi$ gives a counter example to Gromov-Yomdin type inequality at the same time.
	
	For example, let $T$ be the Dynkin diagram of $A_3$-type and let $n\geq 4$ be an even integer. We label the vertices of $T=A_3$ as $V(A_3) = \{1,2,3\}$ so that the vertex  $1$ (resp.\ $3$) is connected to the vertex $2$. Let $\phi$ be the symplectic automorphism given as 
	\[\phi = \tau_1^2 \circ \tau_2^{-2} \circ \tau_3 ^2: P_n(A_3) \stackrel{\simeq}{\to} P_n(A_3).\]
Then it is an element of the Torelli subgroup and therefore the spectral radius of the associated linear map on the Grothendieck group is 1 as explained above. However, Theorem \ref{thm:entropyofpennertype} again says that the categorical entropy of $\phi$ is equal to
\[h_{cat}(\phi) = \log \text{Rad}M_\phi = \log \text{Rad}
	\begin{pmatrix}
		5 & 10 & 4 \\
		2 & 5 & 2 \\	
		0 & 2 & 1
	\end{pmatrix} = \log\left(5+2\sqrt{6}\right).\]	
	
	\begin{rmk}
		\label{rmk torelli}
		Recently, in \cite{Bae-Lee23}, it is proven that every Penner type symplectic automorphism has positive categorical entropy. 
		Moreover, \cite{Bae-Lee22} proves that if $\phi$ is a compactly supported symplectic automorphisms on a Weinstein manifold, then the categorical entropy of $\phi$ bounds the topological entropy from below. 
		Thus, the above argument gives infinitely many examples of symplectic automorphisms $\phi$ such that $\phi$ is an element of Torelli subgroup, but the topological entropy of $\phi$ is positive. 
	\end{rmk}

	\bibliographystyle{amsalpha}
	\bibliography{entropy.bib}
	
\end{document}